\documentclass[psamsfonts,oneside,reqno]{amsart}
\usepackage[utf8]{inputenc}
\usepackage{amsfonts}
\usepackage{hyperref}
\hypersetup{
pdftitle={Small Ball Probabilities for Simple Random Tensors},
pdfauthor={Grigoris Paouris and Xuehan Hu}
}
\usepackage{amsmath}
\usepackage{xcolor}
\usepackage{amsthm}
\usepackage{pdflscape}
\usepackage{pgfplots}
\usepackage{mathrsfs}

\usepackage[utf8]{inputenc}

\newtheorem{theorem}{Theorem}[section]
\newtheorem{lemma}[theorem]{Lemma}

\newtheorem{proposition}[theorem]{Proposition}

\newtheorem{corollary}[theorem]{Corollary} 
\theoremstyle{definition}
\newtheorem{definition}[theorem]{Definition}

\theoremstyle{remark}
\newtheorem{remark}[theorem]{Remark}

\newcommand{\norm}[1]{\left\lVert#1\right\rVert} 
\newcommand{\dd }[1]{\mathrm{d}#1}

\title{Small Ball Probabilities for Simple Random Tensors}

\author{Xuehan Hu$^1$ \quad\quad Grigoris Paouris$^2$ }
\thanks{1\quad Texas A\&M University. Email: huxuehan@tamu.edu. This work is partially supported by NSF grant CCF-1900881 and Simon's grant 964286.}
\thanks{2\quad Texas A\&M University. Email: grigoris@tamu.edu. This work is partially supported by NSF grant CCF-1900881 and Simon's grant 964286.}
\begin{document}

\begin{abstract}
We study the small ball probability of an order-$\ell$ simple random tensor $X=X^{(1)}\otimes\cdots\otimes X^{(\ell)}$ where $X^{(i)}, 1\leq i\leq\ell$ are independent random vectors in $\mathbb{R}^n$ that are log-concave or have independent coordinates with bounded densities. We show that the probability that the projection of $X$ onto an $m$-dimensional subspace $F$ falls within an Euclidean ball of length $\varepsilon$ is upper bounded by $\frac{\varepsilon}{(\ell-1)!}\left(C\log\left(\frac{1}{\varepsilon}\right)\right)^{\ell}$ and also this upper bound is sharp when $m$ is small. We also established that a much better estimate holds true for a random subspace.
\end{abstract}

\maketitle

\section{Introduction}\label{intro}
Tensor decomposition is a crucial problem in learning many latent variable models in data science, such as mixture models (see \cite{anandkumar2012method}, \cite{smoothed}), hidden Markov models (see \cite{anandkumar2012method}, \cite{chang1996full}, \cite{hsu2012spectral}), phylogenetic reconstruction (see \cite{comon1994independent}, \cite{mossel2005learning}) and so on. Many results and algorithms have been developed in the framework where the component vectors of the tensors are noisy, known as smoothed analysis model (see \cite{smoothedmodel}). Bhaskara et al \cite{smoothed} established algorithms to decompose fixed-rank random tensors based on Chang's lemma whose efficiency and robustness have been reduced to finding the small ball estimate for simple random tensors. The study of anti-concentration of random tensors is primarily inspired by this problem. 

In recent years, there has been a plethora of results on the concentration of random tensors. Several papers discussed the concentration inequalities for the projection of simple random tensor onto any fixed direction. Latala \cite{latala} gives an estimate on the case where the component vectors of the simple tensor are independent standard Guassian random vectors, of which Lehec \cite{Lehec2011} provides another proof based on Talagrand’s majorizing theorem. Later there are papers that considered more general distributions for the component vectors: Adamczak and Latala \cite{AdamczakandLatala} considered log-concave distribution; Adamczak and Wolff \cite{Adamczak2013ConcentrationIF} considered distribution that satisfies a Sobolev inequality; G\"{o}tze, Sambale, and Sinuli \cite{Gotze2019ConcentrationIF}
considered $\alpha$-subexponential distribution; Vershynin \cite{Vershynin2019ConcentrationIF} considered subgaussian distribution with optimal dependence on the degree of the tensors, of which Bamberger, Krahmer, Ward \cite{bamberger2021hansonwright} gave improved dependence on the dimension of the component vectors; and Adamczak, Latala, and Melle \cite{AdamczakLatalaMelle} extended the result to values in some Banach spaces. There are more applications of concentration inequalities for simple tensors: Jin, Kolda, Ward \cite{KroneckerFJLT} introduced a faster Johnson-Lindenstrauss projection for embedding vectors with Kronecker product structure; Bamberger, Krahmer, Ward \cite{KroneckerFJLToptimal} gave optimal Johnson–Lindenstrauss property for the aforementioned embeddings.

On the other hand, the anti-concentration is also an important part of contemporary probability and has strong connections with the metric-algebraic structure of the underlying space (see for example the survey \cite{LI2001533} for processes or \cite{TaoVu} \cite{Nguyen2013} for the Littwood-Offord problems). While anti-concentration of random variables and random vectors (see for example \cite{smallball} and \cite{linearimage}) have been extensively studied, that of random tensors remains unclear in many aspects. Small ball probabilities for simple random tensors have been established by Bhaskara et al \cite{smoothed}, Anari et al \cite{tensordecomp}, Glazer and Miklincer \cite{GLAZER2022109639}. Our main result builds on this direction by providing a sharp, small ball probability for random tensors under minimal assumptions on the randomness.

In this paper, we constructed anti-concentration results for simple random tensors under minimal assumptions. It is well-known that for continuous distributions, being anti-concentrated is equivalent to having bounded densities. We provided sharp small ball estimates for the projection of the tensor product of independent log-concave distributions onto a subspace of fixed dimension. And we used symmetrization techniques to extend the results to distributions with independent coordinates of bounded densities, which is the minimal assumption we may have. We also established that a much better estimate holds true for a random projected subspace, i.e. the random subspace in the Grassmanian with respect to the Haar measure on the real orthogonal group.

We refer to \cite{lim2021tensors} ,\cite{Landsberg2011TensorsGA} for detailed definition of tensors. A tensor $X\in\mathbb{R}^{n_1\otimes\cdots\otimes n_{\ell}}$ can be represented as a multidimensional array, that is
    $$X=\left(X_{i_1\cdots i_{\ell}}\right)_{i_1\cdots i_{\ell}}.$$
In particular, $X$ is a simple tensor if there exist vectors $X^{(j)}=\left(X^{(j)}_1,\cdots,X^{(j)}_n\right)\in\mathbb{R}^{n_j}, 1\leq j\leq\ell$, such that
    $$X=X^{(1)}\otimes\cdots\otimes X^{(\ell)}=\otimes_{j=1}^{\ell}X^{(j)}=\left(X_{i_1}^{(1)}\cdots X_{i_{\ell}}^{(\ell)}\right)_{i_1\cdots i_{\ell}}.$$
$X\in\mathbb{R}^{n_1\otimes\cdots\otimes n_{\ell}}$ is also a multilinear map $X:\mathbb{R}^{n_1\times\cdots\times n_{\ell}}\rightarrow\mathbb{R}$, such that for $Y\in\mathbb{R}^{n_1\otimes\cdots\otimes n_{\ell}}$,
    $$\left\langle X,Y\right\rangle:=\sum_{i_1\cdots i_{\ell}}X_{i_1\cdots i_{\ell}}Y_{i_1\cdots i_{\ell}},$$
which is usually refered to as the Frobenius inner product of tensors. And 
    $$\norm{X}_2:=\sqrt{\left\langle X,X\right\rangle}$$
is the Frobenius norm of a tensor, which aligns with the Euclidean norm when we view the tensor in $\mathbb{R}^{n_1\otimes\cdots\otimes n_{\ell}}$ as a flattened vector in $\mathbb{R}^{n_1\times\cdots\times n_{\ell}}$. In particular, if 
    $$X=\otimes_{j=1}^{\ell}X^{(j)}, Y=\otimes_{j=1}^{\ell}Y^{(j)},$$
then
    $$\left\langle X,Y\right\rangle:=\prod_{j=1}^{\ell}\left\langle X^{(j)},Y^{(j)}\right\rangle$$
and 
    $$\norm{X}_2=\prod_{j=1}^{\ell}\norm{X^{(j)}}_2.$$

Recall that a vector has a unique orthogonal decomposition with respect to a linear subspace. For any tensor $X\in\mathbb{R}^{n_1\otimes\cdots\otimes n_{\ell}}$ and subspace $F\subset\mathbb{R}^{n_1\otimes\cdots\otimes n_{\ell}}$, there exists a unique tensor $\Pi_FX$ in $F$ such that $\norm{X-\Pi_FX}$ is minimized. $\Pi_FX$ is called the orthogonal projection of $X$ onto $F$.
 
In the paper, we use the letters $C,C',C_1,C_2,C_3,...$ to represent universal constants that may be different from line to line. We haven't tried to optimize these constants. One can compute all the constants explicitly, which is not the main purpose of this paper. Our main result reads as follows:

\begin{theorem}\label{1.1}
    Let $X^{(j)}\in\mathbb{R}^{n_j}, 1\leq j\leq\ell$ be independent random vectors with independent coordinates whose densities have uniform norms bounded by some constant $M>0$. Suppose $F$ is a subspace in $\mathbb{R}^{n_1\otimes\cdots\otimes n_{\ell}}$ with dimension $m$ and suppose $z_j\in\mathbb{R}^{n_j}, 1\leq j\leq\ell$ are arbitrary vectors, then for $0<\varepsilon<e^{-c{\ell}}$,
        $$\mathbb{P}\left(\norm{\Pi_F\otimes_{j=1}^{\ell}\left(X^{(j)}-z_j\right)}_2\leq \frac{1}{(CM)^{\ell}}\varepsilon\sqrt{m}\right)\leq\min\left\{m,{C'}^{\ell}\log\frac{1}{\varepsilon}\right\}\frac{\varepsilon}{(\ell-1)!}\left(C''\log\frac{1}{\varepsilon}\right)^{\ell-1}.$$
\end{theorem}

The above estimate cannot be improved (see Appendix \ref{A} for the details). In fact, let $X^{(1)},\cdots,X^{(\ell)}$ be independent uniform distributions on $[-\sqrt{3},\sqrt{3}]^n$ such that $\mathbb{E}\left[X^{(j)}_i\right]=0$ and $\mathbf{Var}\left(X_i^{(j)}\right)=1$ for $1\leq i\leq n, 1\leq j\leq\ell$. Then for any $1\leq m\leq n$, there exists a subspace $F$ of dimension $m$ such that
\begin{align*}
    \mathbb{P}\left(\norm{\Pi_FX^{(1)}\otimes\dots\otimes X^{(\ell)}}_2\leq\varepsilon\sqrt{m}\right)\geq\frac{C\varepsilon}{(\ell-2)!}\left(\log{\frac{1}{\varepsilon}}\right)^{\ell-2}.
\end{align*}

It is easy to check (see Section 1 in \cite{linearimage}) that for continuous distributions, to be anti-concentrated is equivalent to have bounded densities. In this sense, for continuous distributions, the densities to be bounded are practically the minimal assumptions that we can pose to the problem.

We observe that the orthogonal projection of $\otimes_{j=1}^{\ell}X^{(j)}$ onto subspaces with the same dimension can have quite different small ball behaviors. When the subspace is generic, then with high probability we are able to obtain much better estimates. Consider the orthogonal group $\mathbb{O}\left(n^{\ell}\right)$ equipped with the unique Haar probability measure invariant under the action of the group (See Chapter 1 in \cite{elizabeth}). In the rest of the paper we denote by $G_{n,m}$ the Grassmannian manifold of $m$-dimensional subspaces in $\mathbb{R}^n$ equipped with the Haar probability measure invariant under the action of orthogonal group.
We have the following theorem.

\begin{theorem}\label{1.2}
    There exists a subset $\mathcal{S}_{\mathcal{F}}$ in $G_{n^{\ell},m}$ with Haar measure at least $1-e^{-c\max\left\{m,n\right\}}$. Let $X^{(1)},\cdots,X^{(\ell)}\in\mathbb{R}^n$ be independent random vectors with independent coordinates whose densities have uniform norms bounded by some constant $M>0$. Let $z_1,\cdots,z_{\ell}\in\mathbb{R}^{n}$ be arbitrary vectors and let $m\leq n^{\ell}$. Then for every subspace $F\in\mathcal{S}_{\mathcal{F}}$ with dimension $m$, then for $0<\varepsilon<1$,
    \[ \mathbb{P}\left(\norm{\Pi_F\otimes_{j=1}^{\ell}\left(X^{(j)}-z_j\right)}_2\leq\frac{1}{(CM)^{\ell}}\varepsilon\sqrt{m}\right)\leq 
          (C'\varepsilon)^{C''\min\left\{m,n\right\}}+e^{-C'''n}.
    \]
\end{theorem}

The proof of \textsc{Theorem} \ref{1.1} relies on a stochastic dominance argument. We show that the small ball behavior of tensor product of independent random vectors with independent coordinates and bounded densities is dominated by that of tensor product of uniform random vectors on the cube. In fact we have the following more general theorem. This line of research builds on a series of papers on empirical isoperimetric inequalities and applications to small ball estimates. See Paouris, Pivovarov \cite{randomized}, \cite{RandomConvexSets} and Dann, Paouris, Pivovarov \cite{BoundingMarginal}.

\begin{theorem}\label{1.3}
    Let $X^{(j)}\in\mathbb{R}^{n_j}, 1\leq j\leq\ell$ be independent random vectors with independent coordinates whose densities have uniform norms bounded by 1, then for any symmetric convex body $K\subset\mathbb{R}^{n_1\otimes\cdots\otimes n_{\ell}}$, we have
    $$\mathbb{P}\left(\otimes_{j=1}^{\ell} X^{(j)}\in K\right)\leq\mathbb{P}\left(\otimes_{j=1}^{\ell}\mathbf{1}_{{\left[-\frac{1}{2},\frac{1}{2}\right]}^{n_j}}\in K\right).$$
    Here by an abuse of notation, we let $\mathbf{1}_{\left[-\frac{1}{2},\frac{1}{2}\right]^{n_j}}$ denotes  the  uniform  distribution  on $\left[-\frac{1}{2},\frac{1}{2}\right]^{n_j}$.
\end{theorem}

With \textsc{Theorem} \ref{1.2} it suffices to study small ball probabilities for tensor product of random vectors with independent coordinates uniformly distributed on a centered interval. Instead of working with these particular distributions, we will work with a general log-concave isotropic distribution. (See Section \ref{Preliminaries} for precise definition).

\begin{theorem}\label{1.4}
    Let $X^{(j)}\in\mathbb{R}^{n_j}, 1\leq j\leq\ell$ be independent log-concave isotropic random vectors. Suppose $F$ is a subspace in $\mathbb{R}^{n_1\otimes\cdots\otimes n_{\ell}}$ with dimension $m$ and suppose $f^{(1)},\cdots,f^{(m)}$ is an orthonormal basis for $F$. Then for $0<\varepsilon<e^{-c{\ell}}$,
        $$\mathbb{P}\left(\left|\langle X^{(1)}\otimes\dots\otimes X^{(\ell)},f^{(k)} \rangle\right|\leq\varepsilon\right)\leq \frac{\varepsilon}{(\ell-1)!}\left(C\log\frac{1}{\varepsilon}\right)^{\ell-1},$$
    and
        $$\mathbb{P}\left(\norm{\Pi_FX^{(1)}\otimes\dots\otimes X^{(\ell)}}_2\leq\varepsilon\sqrt{m}\right)\leq\min\left\{m,{C'}^{\ell}\log\frac{1}{\varepsilon}\right\}\frac{\varepsilon}{(\ell-1)!}\left(C\log\frac{1}{\varepsilon}\right)^{\ell-1}.$$
\end{theorem}
Note that in \textsc{Theorem} \ref{1.4}, in contrast to \textsc{Theorem} \ref{1.1}. we do not require the coordinates of each component vector be independent.
\begin{theorem}\label{1.5}
    Let $X^{(1)},\cdots,X^{(\ell)}\in\mathbb{R}^n$ be independent isotropic log-concave random vectors. Let $z_1,\cdots,z_{\ell}\in\mathbb{R}^{n}$ be arbitrary vectors and let $m\leq n^{\ell}$. Then there exists a subset $\mathcal{S}_{\mathcal{F}}$ in $G_{n^{\ell},m}$ with Haar measure at least $1-e^{-c\max\left\{m,n\right\}}$. For every subspace $F\in\mathcal{S}_{\mathcal{F}}$ with dimension $m$,  then for $0<\varepsilon<1$, 
    \[ \mathbb{P}\left(\norm{\Pi_F\otimes_{j=1}^{\ell}\left(X^{(j)}-z_j\right)}_2\leq \varepsilon\sqrt{m}\right)\leq 
          (C\varepsilon{\mathcal{L}}_{C'\min\{m,n\}})^{C'\min\left\{m,n\right\}}+e^{-\frac{C''\sqrt{n}}{C_P}}.
    \]
    where the isotropic constant
        $${\mathcal{L}}_{r}:=\sup_F\mathcal{L}_{\Pi_{F}}(\mu)=\sup_F\norm{f_{\Pi_FX}(x)}_{\infty}^{1/r},$$
    where $F$ is an $r$-dimensional subspace of $\mathbb{R}^n$ and $f_{\Pi_FX}(x)$ denotes the density of the marginal distribution $\Pi_FX$. And $C_P$ denotes the maximum over the Poincar\'e constants of $X^{(1)},\cdots,X^{(\ell)}$.
\end{theorem}
Acknowledgments: The second named author is grateful to J.M. Landsberg and Liza Rebrova for several interesting discussions and references. The second named author is grateful to the Mathematics Department of Princeton for its hospitality, where part of this work was carried out. And we are grateful to Dan Mikluncer and Alperen A. Erg\"ur for many helpful comments.

\section{Background and Related Results}

\subsection{Motivation}
The study of small ball probabilities for random tensors is primarily inspired by the tensor decomposition problem. Tensor decomposition is an important question in many latent variable models, such as multi-view model, Gaussian mixture model (\cite{smoothed}), Hidden Markov model (\cite{Chang1996FullRO}, \cite{MosselRoch}, \cite{pmlr-v23-anandkumar12}, \cite{hsu2012spectral}) and assembly of neurons (\cite{tensordecomp}). The tensor rank of $X$ is the smallest number $r$ such that 
    $$X=\sum_{i=1}^rX_i^{(1)}\otimes\dots\otimes X_i^{(\ell)}$$
is a sum of $r$ simple tensors. For a fixed rank-$r$ tensor $X$, retrieving all the component vectors $X_i^{(j)}$'s is an NP-hard problem in the worst case scenario. One method is to apply the smoothed analysis model introduced by Spielman and Teng (see \cite{smoothedmodel}). That is to say we introduce random noises to $X_k^{(j)}$'s. The following is the smoothed analysis model used by Bhaskara et al in \cite{smoothed}. Consider random vectors 
    $$\widetilde{X}_i^{(j)}=X_i^{(j)}+G_i^{(j)}\in\mathbb{R}^n, 1\leq i\leq r, 1\leq j\leq l$$
where 
    $$\norm{X_i^{(j)}}_2\leq C,\quad G_i^{(j)}\sim N\left(0,\frac{\rho^2}{n}\mathbb{I}_n\right).$$
Here $\widetilde{X}_i^{(j)}$ is the noisy version of $X_i^{(j)}$, and we have normalized it such that in expectation its length has not been increased.
Define an $n\times r$ matrix $\widetilde{A}_j=\left[\widetilde{X}_1^{(j)},\cdots,\widetilde{X}_r^{(j)}\right]$ for $1\leq j\leq\ell$, then the Khatri-Rao product of $\widetilde{A}_j$'s is defined to be an $n^{\ell}\times r$ matrix
    $$\widetilde{A}=\widetilde{A}_1\odot\widetilde{A}_2\odot\cdots\odot\widetilde{A}_{\ell},$$
where the $i$-th column of $\widetilde{A}$ is the $n^{\ell}$-dimensional flattened vector $\otimes_{j=1}^{\ell}\widetilde{X}_k^{(i)}$. Then we can rewrite the rank-$r$ tensor $X$ using Katri-Rao product as
    $$X=\sum_{i=1}^r\left(\widetilde{A}_1\odot\cdots\odot\widetilde{A}_{\lfloor\frac{\ell-1}{2}\rfloor}\right)\otimes\left(\widetilde{A}_{\lfloor\frac{\ell-1}{2}\rfloor+1}\odot\cdots\odot\widetilde{A}_{\ell-1}\right)\otimes \widetilde{A}_{\ell}.$$
If the columns of the matrices $\left(\widetilde{A}_1\odot\cdots\odot\widetilde{A}_{\lfloor\frac{\ell-1}{2}\rfloor}\right)$ and $\left(\widetilde{A}_{\lfloor\frac{\ell-1}{2}\rfloor+1}\odot\cdots\odot\widetilde{A}_{\ell-1}\right)$ are robustly linearly independent respectively, then Bhaskara et al show that there is an efficient algorithm to retrieve all the $\widetilde{X}_i^{(j)}$'s. The algorithm is known as simultaneous diagonalization (\cite{orderthree}) or Chang's lemma (\cite{Comon1994IndependentCA}).

We want to show that columns of $\widetilde{A}$ are robustly linearly independent, or $\widetilde{A}$ is well "invertible". This is equivalent to bounding the smallest singular value of the Khatri-Rao product. And by \textsc{Lemma} \ref{smin} the smallest singular value is closely related to the orthogonal projections of a column vector onto the orthogonal complement of the rest of the column vectors. The investigation of the papers Bhaskara et al \cite{smoothed} and Anari et al \cite{tensordecomp} reduces the quantification of the algorithm (running time, probability of failure, and size of the noise compared to the size of data) to the problem of small ball probabilities for orthogonal projection of simple tensors.

\subsection{Preliminaries}\label{Preliminaries}
To study the small ball probabilities for the orthogonal projection of simple random tensors, we first need to decide what randomness we care about. In \cite{smoothed}, the authors considered Gaussian perturbations. We want to extend the independent Gaussian random vectors to independent random vectors with independent coordinates whose densities are universally bounded. Another setting we care about is when the random vectors are isotropic and log-concave.

A random vector in $\mathbb{R}^n$ is log-concave if its density $f$ is log-concave, i.e. for $x,y$ in the support of $f$ and $\theta\in(0,1)$, we have
        $$f(\theta x+(1-\theta)y)\geq f(x)^{\theta}f(y)^{1-\theta}.$$
Prékopa–Leindler inequality (see for example \cite{AGGbook1}) implies sum of log-concave random vectors is log-concave. And affine linear map of log-concave random vector is also log-concave (see \textsc{Lemma} 2.1 in \cite{dharmadhikari1988unimodality}). Examples of log-concave vectors are the Gaussians, exponentials and uniform measures on convex bodies.
A random vector in $X\in\mathbb{R}^n$ is isotropic if
        $$\mathbb{E}\left[XX^T\right]=Id.$$
Any random vector with second moments has a linear image that is isotropic.
Log-concave measure has been extensively investigated in the last decades. Concentration questions for these measures were (and still are) some of the most major open questions in asymptotic geometric analysis (see \cite{AGGbook1}, \cite{AGGbook2}, \cite{Isotropic}). Two of the most important open questions are Bourgain's hyperplane conjecture (see \cite{bourgain1986high}, \cite{bourgain1986geometry}) and Kannan Lovasz Simonovitch's Question (see \cite{kannan1995isoperimetric}). We refer to Klartag and Milman \cite{kls1} and Klartag \cite{klartag2023logarithmic} for the history of the problem and the latest developments. 

If $X$ is an isotropic log-concave random vector in $\mathbb R^{n}$ then the isotropic constant is $${\mathcal{L}}_{X} := \|f_{\mu}\|_{\infty}^{\frac{1}{n}}. $$ 
The KLS constant (or Poincar\'e constant) of $ X$ (or of $\mu$ that $X$ is distributed) is the smallest $C_P$ such that for every $f$ smooth function on $\mathbb R^{n}$ 
$$ {\rm Var}(f(X) ) \leq C_P \mathbb E \| \nabla f(X) \|_{2}^{2} . $$
Bourgain's slicing problem asks (equivalently) if $ {\mathcal{L}}_{X}$ is uniformly bounded by a constant for every $X$ isotropic log-concave (independent of the dimension). KLS question can be expressed (equivalently) as if $C_{P}$ can be uniformly bounded by a constant for every $X$ isotropic log-concave. The best up-to-date bounds are 
$$ \max\{ {\mathcal{L}_{X} } , \sqrt{C_{P}(X)} \} \leq C \sqrt{\log{n}} $$
due to Klartag \cite{klartag2023logarithmic}. It is well known that Poincar\'e inequality implies exponential concentration (see \cite{Ledoux2001TheCO}), i.e. 
for every $f: \mathbb R^{n} \rightarrow \mathbb{R}$ $1$-Lipschitz, then for $0<\varepsilon<1$,
    $$\mathbb{P}\left(\left|f(X)-\mathbb{E}\left[f(X)\right]\right|\geq t\right)\leq 2e^{-\frac{Ct}{C_P}}.$$
Estimates for large deviations and small ball probabilities for the Euclidean norm are also known as follows (see \cite{Paouris2006ConcentrationOM}, \cite{smallball}). 
\begin{theorem}\label{Paouris}
    Let $X\in\mathbb{R}^n$ be isotropic log-concave random vector and let $q\geq1$, then
        $$\left(\mathbb{E}\norm{X}_2^q\right)^{\frac{1}{q}}\leq C\left(\sqrt{n}+q\right).$$
\end{theorem}
\begin{theorem}\label{PaourisSBP}
    Let $X\in\mathbb{R}^n$ be isotropic log-concave random vector, then for $0<\varepsilon<1$,
        $$\mathbb{P}\left(\norm{X}_2\leq\varepsilon\sqrt{n}\right)\leq\varepsilon^{c\sqrt{n}}$$
\end{theorem}

Here we are mostly interested in small ball probabilities. Small ball probabilities for log-concave measures have been investigated by Carbery and Wright \cite{carberywright}, Guedon \cite{guedon}, Fradelizi \cite{Fradelizi}, Nazarov, Sodin and Volberg \cite{nazarov2001local} \cite{nazarov2002geometric} and Glazer and Mikulincer \cite{GLAZER2022109639}.

\begin{theorem}
    (Carbery-Wright, see \cite{carberywright}) Let $p:\mathbb{R}^n\longrightarrow X$ be a polynomial of degree at most $d$, let $\mu$ be an isotropic log-concave measure on $\mathbb{R}^n$ and let $0\leq q\leq\infty$. Define the functional $p^\#(x)=\norm{p(x)}^{\frac{1}{d}}$.Then there exists an absolute constant $C$ independent of $p,d,\mu,n,q$ and $X$ so that for any $\alpha>0,$\\
    (a) if $n\leq q$ then 
        $$\norm{p^\#}_q\alpha^{-1}\mu\left(x\in K: p^{\#}(x)\leq\alpha\right)\leq Cn;$$
    (b) if $q\leq n$ then
        $$\norm{p^\#}_q\alpha^{-1}\mu\left(x\in K: p^{\#}(x)\leq\alpha\right) \leq C\max{(q,1)}.$$
\end{theorem}
\begin{theorem}\label{guedonthm}
    (Gu\'edon, see \cite{guedon}) Let $A$ be a symmetric convex body, $\mu$ a log-concave probability measure over $\mathbb{R}^n$, then for all $0<\varepsilon\leq1$
        $$\frac{\mu(\varepsilon A)}{\varepsilon}\leq2\log\left(\frac{1}{1-\mu(A)}\right).$$
\end{theorem}

Finally we state the concentration inequalities for random orthogonal matrices (see \cite{elizabeth}). The orthogonal group $\mathbb{O}(n)$ is the group of orthogonal $n\times n$ matrices. Then there exists a unique translation invariant probability measure on $\mathbb{O}(n)$, called Haar measure. 
\begin{theorem}\label{Lipschitz}
    (See \textsc{Theorem} 5.5 in \cite{elizabeth}) Suppose $X\in\mathbb{O}(n)$ is a random orthogonal matrix distributed according to Haar measure. If $\mathcal{A}:\left(\mathbf{Mat}_{n\times n}\left(\mathbb{R}\right),\norm{\cdot}_{HS}\right)\rightarrow\left(\mathbb{R},|\cdot|\right)$ is $1$-Lipschitz with $\mathbb{E}\left[\mathcal{A}\left(X\right)\right]<\infty$, then for every $t>0$,
        $$\mathbb{P}\left(\left|\mathcal{A}(X)-\mathbb{E}\left[\mathcal{A}(X)\right]\right|\geq t\right)\leq2e^{-Cnt^2}.$$
\end{theorem}
The following is a direct corollary of \textsc{Theorem} \ref{Lipschitz}.
\begin{corollary}\label{moment}
    Suppose $X\in\mathbb{O}(n)$ is a random orthogonal matrix distributed according to Haar measure. If $\mathcal{A}:\left(\mathbf{Mat}_{n\times n}\left(\mathbb{R}\right),\norm{\cdot}_{HS}\right)\rightarrow\left(\mathbb{R},|\cdot|\right)$ is $1$-Lipschitz with $\mathbb{E}\left[\mathcal{A}\left(X\right)\right]<\infty$, then for every $q\geq1$,
        $$\mathbb{E}\left|\mathcal{A}(X)-\mathbb{E}\left[\mathcal{A}(X)\right]\right|^q\leq\left(\frac{Cq}{n}\right)^{\frac{q}{2}},$$
    and
        $$\mathbb{E}\left[\mathcal{A}\left(X\right)^2\right]\leq\left(\mathbb{E}\left[\mathcal{A}(X)\right]\right)^2+\frac{C}{n}.$$
    where $C>1$ is universal constant.
\end{corollary}

\subsection{Stochastic Dominance}
We will use stochastic dominance to prove the anti-concentration of simple random tensors when the component vectors are independent with independent coordinates whose densities are universally bounded. We introduce the following notions and definitions from \cite{lieb2001analysis}, \cite{Burchard2009ASC}, and \cite{randomized}.

\begin{definition}
    Let $A$ be a Borel subset of $\mathbb{R}^n$ with finite Lebsegue measure. Then the symmetric rearrangement $A^*$ of $A$ is the open ball centered at the origin, whose volume is equal to the measure of $A$. 
    Let $f:\mathbb{R}^n\rightarrow \mathbb{R}_+$ be a Borel measurable function such that $\{x\in \mathbb{R}^n:f(x)>t\}$ has finite measure for every $t>0$. Then the symmetric decreasing rearrangement $f^*$ of $f$ is defined by
        $$f^*(x)=\int_0^{\infty}\mathbf{1}_{\{f>t\}^*}(x)dt.$$
\end{definition}

\begin{definition}
    Let $X_1$, $X_2$ be random vectors in $\mathbb{R}^n$. Then we say that $X_1$ is less peaked than $X_2$, or $X_1$ is stochasically dominated by $X_2$, denoted by
        $$X_1\prec X_2$$
    if
        $$\mathbb{P}(X_1\in K)\leq \mathbb{P}(X_2\in K)$$
    for any symmetric convex body $K\subset\mathbb{R}^n.$
\end{definition}

\begin{lemma}
    Suppose X is a random variable and  $f:\mathbb{R}\longrightarrow\mathbb{R}_+$ is its density, such that
        $$\norm{f}_1=1, \norm{f}_{\infty}\leq1,$$
    and $f$ is even, then
        $$X^*\prec\mathbf{1}_{\left[-\frac{1}{2},\frac{1}{2}\right]}.$$
    Here by an abuse of notation, we let $1_{\left[-\frac{1}{2},\frac{1}{2}\right]}$ denotes the uniform distribution on $\left[-\frac{1}{2},\frac{1}{2}\right]$.
\end{lemma}

\begin{remark}
    Here the evenness of $f$ is not necessary. In fact, suppose $f$ satisfies that
        $$\norm{f}_1=1, \norm{f}_{\infty}\leq1,$$
    and note that 
        $$\norm{f}_p=\norm{f^*}_p,\quad 1\leq p\leq\infty,$$
    and $f^*$ is even, then
        $$f^{**}=f^*\prec\mathbf{1}_{\mathbb{Q}}.$$
\end{remark}

\begin{theorem}\label{BLL}
    (Rogers-Brascamp-Lieb-Luttinger inequality, see \cite{Rogers} \cite{BLL}) Let $\mu$ be a quasi-concave measure in $\mathbb{R}^n$ supported in a symmetric convex set $K$. Let $u_1,\dots,u_m$ be non-zero vectors in $\mathbb{R}^n$. Let $f_1,\dots,f_m$ be measurable non-negative integrable functions on $\mathbb{R}$. Then
        $$\int_{\mathcal{R}^n}\prod_{i=1}^mf_i(\langle x,u_i\rangle)d\mu(x)\leq\int_{\mathcal{R}^n}\prod_{i=1}^mf_i^*(\langle x,u_i\rangle)d\mu(x).$$
\end{theorem}
A function $f$ is unimodal if f it is the increasing limit of a sequence of functions of the form,
    $$\sum_{i=1}^mt_i\mathbf{1}_{K_i}$$
where $t_i>0$ and $K_i$ are symmetric convex bodies in $\mathbb{R}^n$. For every integrable function $f:\mathbb{R}^n\rightarrow\mathbb{R}$, its symmetric decreasing rearrangement $f^*$ is unimodal.
\begin{theorem}\label{Kanter}
    (Kanter, see \cite{unimodality}) Let $f_1$, $g_1$ be functions on $\mathbb{R}^{n_1}$ such that $f_1\prec g_1$ and $f$ a unimodal function on $\mathbb{R}^{n_2}$, then
        $$ff_1\prec fg_1.$$
    In particular. if $f_i, g_i$ are unimodal functions on $\mathbb{R}^{n_i}, i=1,\cdots,M$ and $f_i\prec g_i$ for all $i$, then
        $$\prod_{i=1}^Mf_i\prec\prod_{i=1}^Mg_i.$$
\end{theorem}

\subsection{Related Results}
Since there is no good way to characterize the span of $m$ simple random tensors, most known results consider a general $m$-dimensional space (see \cite{smoothed}, \cite{tensordecomp}). Let $F$ be an  $m$-dimensional space in $\mathbb{R}^{n^{\ell}}$. To study the small ball behavior of $\Pi_F\otimes_{j=1}^{\ell}X^{(j)}$, we choose $f^{(1)},\cdots,f^{(m)}$ to be an orthonormal basis for $F$. Then
\begin{align*}
    \norm{\Pi_FX^{(1)}\otimes\dots\otimes X^{(\ell)}}_2^2&=\sum_{k=1}^m\left|\left\langle X^{(1)}\otimes\dots\otimes X^{(\ell)},f^{(k)}\right\rangle_F\right|^2\\
    &=\sum_{k=1}^m\left|\sum_{i_1,\dots,i_{\ell}}X_{i_1}^{(1)}\dots X_{i_{\ell}}^{(\ell)}f_{i_1\dots i_{\ell}}^k\right|^2.
\end{align*}

Note that $p=||\Pi_FX^{(1)}\otimes\dots X^{(\ell)}||_2^2$ is a polynomial of degree $2\ell$. By Carbery-Wright inequality, we have
    $$\mathbb{P}\left(\norm{\Pi_FX^{(1)}\otimes\dots\otimes X^{(\ell)}}_2<\varepsilon\sqrt{m}\right)\leq C\ell\varepsilon^{1/l}.$$

The following result is from Anari et al. It applies to a very broad setting of random vectors including discrete ones. The result reads as follows (see \textsc{Lemma} 6 in \cite{tensordecomp}).

\begin{theorem}
    A random vector $X\in\mathbb{R}^n, 1\leq i\leq\ell$ is drawn from $(\varepsilon,p)$-nondeterministic distribution if for every $j\in[n]$ and any interval of the form $(t-\varepsilon,t+\varepsilon)$, we have
        $$\mathbb{P}\left(X_j\in(t-\delta,t-\delta)\mid X_{-j}\right)\leq p,$$
    where $X_j$ represents the projection of $X$ onto the coordinates $[n]-{j}$. Assume the $n$-dimensional vectors $X^{(1)},\dots,X^{(\ell)}$ are drawn according to an $(\varepsilon,p)$-nondeterministic distribution. Further assume that $V\subset\mathbb{R}^{n^{\otimes l}}$ is a subspace of dimension at most $(cn)^{\ell}$. Then
        $$\mathbb{P}\left(dist\left(X^{(1)}\otimes\dots\otimes X^{(\ell)},V\right)<\left(\frac{\varepsilon}{\sqrt{n}}\right)^{\ell}\right)\leq2n^{\ell-1}p^{(1-c)n}.$$
\end{theorem}
Note that if a random vector has independent coordinates with bounded densities (by $1$) then it is $(\varepsilon, c\varepsilon)$-nondeterministic for any $\varepsilon$ in $(0,1)$. Since $\norm{\Pi_FX^{(1)}\otimes\cdots\otimes X^{(\ell)}}_2=dist(X^{(1)}\otimes\dots\otimes X^{(\ell)},F^{\perp})$. Then their result under our assumptions implies the following estimate
    $$\mathbb{P}\left(\norm{\Pi_F(X^{(1)}\otimes\dots\otimes X^{(\ell)})}_2<\varepsilon\right)\leq2n^{\ell-1}\left(c\sqrt{n}\right)^{n-(n^{\ell}-m)^{1/\ell}}\varepsilon^{\frac{n-(n^{\ell}-m)^{1/\ell}}{\ell}}.$$
If $n<\ell$, or if $n\geq\ell$ and $m<n^{\ell}-(n-\ell)^{\ell}$, then $\frac{n-(n^{\ell}-m)^{1/\ell}}{\ell}<1.$

Small ball estimates can also come from concentration inequalities of simple random tensors. The following results are from Vershynin \cite{Vershynin2019ConcentrationIF} and Bamberger, Krahmer, Ward \cite{bamberger2021hansonwright}.
\begin{theorem}\label{VershyninConcentration}
    (See \textsc{Corollary} 6.1 in \cite{Vershynin2019ConcentrationIF}) Let $X^{(1)},\cdots,X^{(\ell)}\in\mathbb{R}^{n}$ be independent random vectors whose coordinates are independent, mean zero, unit variance with subgaussian norm bounded by $L\geq 1$. Suppose $F$ is a subspace in $\mathbb{R}^{n^{\ell}}$ of dimension $m$. Then for $0\leq\varepsilon\leq1$,
        $$\mathbb{P}\left(\norm{\Pi_F\otimes_{j=1}^{\ell}X^{(j)}}_2\leq\varepsilon\sqrt{m}\right)\leq 2\exp\left(-\frac{C\left(1-\varepsilon\right)^2m}{\ell n^{\ell-1}}\right).$$
\end{theorem}
    
\begin{theorem}\label{BambergerConcentration}
    (See \textsc{Theorem} 2.1 in \cite{bamberger2021hansonwright}) Let $X^{(1)},\cdots,X^{(\ell)}\in\mathbb{R}^{n}$ be independent random vectors whose coordinates are independent, mean zero, unit variance, subgaussian random variables. Suppose $F$ is a subspace in $\mathbb{R}^{n^{\ell}}$ of dimension $m$. Then for a constant $C_{\ell}$ that only depends on $\ell$ and for $0\leq\varepsilon\leq1$,
        $$\mathbb{P}\left(\norm{\Pi_F\otimes_{j=1}^{\ell}X^{(j)}}_2\leq\varepsilon\sqrt{m}\right)\leq e^2\exp\left(-C_{\ell}\frac{\left(1-\varepsilon\right)^2m}{n^{\ell-1}}\right).$$
\end{theorem}
Since the above estimates come from concentration inequalities, they do not approach to $0$ as $\varepsilon$ approaches to $0$. Also, they only apply to the case when the dimension of the subspace $m\gg\mathcal{O}_{\ell}\left(n^{\ell-1}\right)$.
\section{Stochastic Dominance}
We are interested in tensor product of independent random vectors with independent coordinates and bounded densities. It is known that random vectors with independent coordinates whose densities are bounded by 1 are stochastically dominated by the uniform random vectors on the unit cube. In this section we will show that tensor product of random vectors preserves this property. It is well known (see for example section 4.1 in \cite{polytope}) that we can approximate any symmetric convex body by circumscribed polytopes.
\begin{lemma}\label{circumscribed}
    (See \cite{polytope}) Let $d\geq n+1$ and denote
    by $P_{d,o}^n(K)$ the set of polytopes in $\mathbb{R}^n$ that are circumscribed out of $K$ with at most $d$ vertices. Let $K\subset\mathbb{R}^n$ be a symmetric convex body, then
        $$\inf_{V\in P_{d,o}^n(K)}\left|K\backslash V\right|\leq\frac{c(K)}{d^{\frac{2}{n-1}}},$$
    where $c(K)$ is a constant that depends only on $K$.
\end{lemma}
\begin{theorem}\label{dominance1}
    Let $X^{(j)}=\left(X_1^{(j)},\cdots,X_{n_j}^{(j)}\right)\in\mathbb{R}^{n_j}, 1\leq j\leq\ell$ be independent random vectors with independent coordinates. Then
        $$\otimes_{j=1}^{\ell} X^{(j)}\prec \otimes_{j=1}^{\ell} {X^{(j)}}^*.$$
\end{theorem}

\begin{proof}
    We will consider the tensor as a flattened vector:
        $$\otimes_{j=1}^{\ell}X^{(j)}=\left(\prod_{j=1}^{\ell}X_{i_j}^{(j)}\right)_{i_1\dots i_{\ell}}.$$

    Let $K$ be a compact and symmetric convex set in $\mathbb{R}^{n^{\ell}}$, then there exists a sequence of polytopes $\left\{K_n\right\}_{n=1}^{\infty}$ satisfying \textsc{Lemma} \ref{circumscribed}, where
        $$K_n=\bigcap_{k=1}^N\{x: |\left\langle x,u^k\right\rangle|\leq1\}$$
    and
        $$x=\otimes_{j=1}^{\ell}X^{(j)}, \quad u^k=\left(u^k_{i_1\dots i_{\ell}}\right)_{i_1\dots i_{\ell}}.$$
    Without loss of generality, we can assume that
        $$K=\bigcap_{l=1}^N\{x: |\langle x,u^k\rangle|\leq1\}.$$
    This is to say
        \begin{equation}\label{eq1}
            \otimes_{j=1}^{\ell}X^{(j)}\in K \iff \left\langle\otimes_{j=1}^{\ell}X^{(j)},u^k\right\rangle\leq1, \quad\forall 1\leq k\leq N.
        \end{equation}
    For every $p=1,\cdots,\ell$ we can write
    \begin{align}
        \langle \otimes_{j=1}^{\ell} X^{(j)},u^k\rangle & =\sum_{i_1,\dots,i_{\ell}}\left(X_{i_1}^{(1)}\cdots X_{i_{\ell}}^{(\ell)}\right)u_{i_1\dots i_{\ell}}^k\nonumber\\
        &=\sum_{i_p}X_{i_p}^{(p)}\sum_{i_1,\cdots,i_{p-1},i_{p+1},\cdots,i_{\ell}}\,\left(X_{i_1}^{(1)}\cdots X_{i_{p-1}}^{(p-1)}X_{i_{p+1}}^{(p+1)}\cdots X_{i_{\ell}}^{(\ell)}\right)u_{i_1\cdots i_{\ell}}^k\nonumber\\
        &=\left\langle X^{(p)},\theta^k_p\right\rangle\label{eq2}
    \end{align}
    where 
        $$\theta^k_p:=\theta^k_p\left(X^{(j)},j\neq p,u^k\right)=\left(\sum_{i_1,\cdots,i_{p-1},i_{p+1},\cdots,i_{\ell}}\,\left(X_{i_1}^{(1)}\cdots X_{i_{p-1}}^{(p-1)}X_{i_{p+1}}^{(p+1)}\cdots X_{i_{\ell}}^{(\ell)}\right)u_{i_1\cdots i_{p-1}i_p i_{p+1}\cdots i_{\ell}}^k\right)_{i_p=1}^{n_p}$$
    is an $n_p$-dimensional random vector independent of $X^{(p)}$. Here the index $i_p$ only appears in $u_{i_1\cdots i_{p-1}i_p i_{p+1}\cdots i_{\ell}}^k$. Combine (\ref{eq1}) and (\ref{eq2}) we have for every $p=1,\cdots,\ell$,
    \begin{equation}\label{eq3}
        \mathbf{1}_K\left(\otimes_{j=1}^{\ell}X^{(j)}\right)=\prod_{k=1}^N\mathbf{1}_{[-1,1]}\left(\left\langle X^{(p)},\theta^k_p\right\rangle\right)
    \end{equation}
    Therefore for every $p=1,\cdots,\ell$,
    \begin{align*}
        &\mathbb{P}\left(\otimes_{j=1}^{\ell}X^{(j)}\in K\right) \\    =&\int_{\mathbb{R}^{n_1}}\dots\int_{\mathbb{R}^{n_{\ell}}}\mathbf{1}_K\left(\otimes_{j=1}^{\ell}X^{(j)}\right)\prod_{i=1}^{n_1}f_i^{(1)}\left(\left\langle X^{(1)},e_i\right\rangle\right)\dd X^{(1)}\dots\prod_{i=1}^{n_{\ell}}f_i^{(\ell)}\left(\left\langle X^{(\ell)},e_i\right\rangle\right)\dd X^{(\ell)}\\  =&\int_{\mathbb{R}^{n_1}}\dots\int_{\mathbb{R}^{n_{\ell}}}\prod_{k=1}^N\mathbf{1}_{[-1,1]}\left(\left\langle X^{(p)},\theta^k_p\right\rangle\right)\prod_{i=1}^{n_1}f_i^{(1)}\left(\left\langle X^{(1)},e_i\right\rangle\right)\dd X^{(1)}\dots\prod_{i=1}^{n_{\ell}}f_i^{(\ell)}\left(\left\langle X^{(\ell)},e_i\right\rangle\right)\dd X^{(\ell)} 
    \end{align*}
    which is an $\ell$-fold integral. By Fubini's theorem, we first choose $p=1$ and consider the single integral about $X^{(1)}$
    \begin{align*}
        &\int_{\mathbb{R}^{n_1}} \prod_{k=1}^N\mathbf{1}_{[-1,1]}\left(\langle X^{(1)},\theta_1\rangle\right)\prod_{i=1}^{n_1}f_i^{(1)}\left(\left\langle X^{(1)},e_i\right\rangle\right)\dd X^{(1)} \\
        \leq&\int_{\mathbb{R}^{n_1}} \prod_{k=1}^N\mathbf{1}_{[-1,1]}\left(\langle X^{(1)},\theta_1\rangle\right)\prod_{i=1}^{n_1}{f_i^{(1)}}^*\left(\left\langle X^{(1)},e_i\right\rangle\right)\dd X^{(1)}.
    \end{align*}
    The inequality follows from Rogers-Brascamp-Lieb-Luttinger inequality (\textsc{Theorem} \ref{BLL}) since the indicator function $\mathbf{1}_{[-1,1]} (\langle X^{(1)},\theta_1\rangle)$ is non-negative and even. Then
    \begin{align*}
        &\mathbb{P}\left(\otimes_{j=1}^{\ell}X^{(j)}\in K\right) \\
        \leq&\int_{\mathbb{R}^{n_1}}\dots\int_{\mathbb{R}^{n_{\ell}}}\prod_{k=1}^N\mathbf{1}_{[-1,1]}\left(\left\langle X^{(p)},\theta^k_p\right\rangle\right)\prod_{i=1}^{n_1}{f_i^{(1)}}^*\left(\left\langle X^{(1)},e_i\right\rangle\right)\dd X^{(1)}\cdots\prod_{i=1}^{n_{\ell}}f_i^{(\ell)}\left(\left\langle X^{(\ell)},e_i\right\rangle\right)\dd X^{(\ell)}
    \end{align*}
    Now we repeat (\ref{eq3}) for $p=2,\cdots,\ell$ and apply Fubini's theorem for the single integrals of $X^{(2)},\cdots,X^{(\ell)}$, we have
    \begin{align*}
        &\mathbb{P}(\otimes_{j=1}^{\ell} X^{(j)}\in K) \\
        \leq&\int_{\mathbb{R}^{n_1}}\dots\int_{\mathbb{R}^{n_{\ell}}}\prod_{k=1}^N\mathbf{1}_{[-1,1]}\left(\langle X^{(p)},u^p\rangle\right)\prod_{i=1}^{n_1}{f_i^{(1)}}^*\left(\left\langle X^{(1)},e_i\right\rangle\right)\dd X^{(1)}\dots\prod_{i=1}^{n_{\ell}}{f_i^{(\ell)}}^*\left(\left\langle X^{(\ell)},e_i\right\rangle\right)\dd X^{(\ell)} \\
        =&\mathbb{P}(\otimes_{j=1}^{\ell} {X^{(j)}}^*\in K) 
    \end{align*}
    where $K$ is any polytope. For a general symmetric convex body $K$, recall from \textsc{Lemma} \ref{circumscribed} that there exists a sequence of polytopes $V_n$ that are circumscribed out of $K$ such that 
        $$\lim_{n\rightarrow\infty}\left|V_n\backslash K\right|=0.$$
    Define $K_n=\cap_{i=1}^nV_i$, then $K_n$ is a polytope as well and $K_n\supset K_{n+1}$ and
        $$\lim_{n\rightarrow\infty}\left|K_n\backslash K\right|=0.$$
    Hence $\lim_{n\rightarrow\infty}K_n=K$ almost everywhere. Define
        $$f_n(x)=\mathbf{1}_{K_n^C}(x).$$
    Then $0\leq f_n(x)\leq f_{n+1}(x)$, and for almost every $x$
        $$\mathbf{1}_{K^C}(x)=\lim_{n\rightarrow\infty}f_n(x).$$
    By monotone convergence theorem, 
        \begin{align*}
            &\lim_{n\rightarrow\infty}\mathbb{P}\left(\otimes_{j=1}^{\ell} X^{(j)}\in K_n\right)=1-\lim_{n\rightarrow\infty}\mathbb{P}\left(\otimes_{j=1}^{\ell} X^{(j)}\in K_n^C\right)\\
            =&1-\lim_{n\rightarrow\infty}\mathbb{P}\left(\otimes_{j=1}^{\ell} X^{(j)}\in K^C\right)=\mathbb{P}(\otimes_{j=1}^{\ell} X^{(j)}\in K)
        \end{align*}
\end{proof}
Following similar strategy but applying Kanter's theorem instead of Rogers-Brascamp-Lieb-Luttinger inequality, we prove the following.
\begin{theorem}\label{dominance2}
    Let $X^{(j)}\in\mathbb{R}^{n_j}, 1\leq j\leq\ell$ be independent random vectors with independent coordinates whose densities have uniform norms bounded by 1, then
        $$\otimes_{j=1}^{\ell} {X^{(j)}}^*\prec\otimes_{j=1}^{\ell}\mathbf{1}_{{\left[-\frac{1}{2},\frac{1}{2}\right]}^{n_j}}.$$
    Here  by  an  abuse  of  notation,  we  let $\mathbf{1}_{\left[-\frac{1}{2},\frac{1}{2}\right]^{n_j}}$ denotes  the  uniform  distribution  on $\left[-\frac{1}{2},\frac{1}{2}\right]^{n_j}$.
\end{theorem}
\begin{proof}
    As in the proof of the previous theorem, it suffices to approximate any symmetric convex body by a polytope. Let $K, u^k,\theta_p^k$ for $k=1,\cdots,$ be defined as in the proof of the previous theorem. Then
    \begin{align*}
        &\mathbb{P}\left(\otimes_{j=1}^{\ell}{X^{(j)}}^*\in K\right) \\
        =&\int_{\mathbb{R}^{n_1}}\dots\int_{\mathbb{R}^{n_{\ell}}}\mathbf{1}_K\left(\otimes_{j=1}^{\ell}X^{(j)}\right)\prod_{i=1}^{n_1}{f_i^{(1)}}^*\left(\left\langle X^{(1)},e_i\right\rangle\right)\dd X^{(1)}\dots\prod_{i=1}^{n_{\ell}}{f_i^{(\ell)}}^*\left(\left\langle X^{(\ell)},e_i\right\rangle\right)\dd X^{(\ell)}\\
        =&\int_{\mathbb{R}^{n_1}}\dots\int_{\mathbb{R}^{n_{\ell}}}\prod_{k=1}^N\mathbf{1}_{[-1,1]}\left(\left\langle X^{(p)},\theta^k_p\right\rangle\right)\prod_{i=1}^{n_1}f_i^{(1)}\left(\langle X^{(1)*},e_i\rangle\right)\dd X^{(1)}\dots\prod_{i=1}^{n_{\ell}}{f_i^{(\ell)}}^*\left(\left\langle X^{(\ell)},e_i\right\rangle\right)\dd X^{(\ell)} 
    \end{align*}
    which is an $\ell$-fold integral. By Fubini's theorem, we first choose $p=1$ and consider the single integral
    \begin{align*}
        &\int_{\mathbb{R}^{n_1}} \prod_{k=1}^N\mathbf{1}_{[-1,1]}\left(\langle X^{(1)},\theta_1\rangle\right)\prod_{i=1}^{n_1}{f_i^{(1)}}^*\left(\left\langle X^{(1)},e_i\right\rangle\right)\dd X^{(1)} \\
        \leq&\int_{\mathbb{R}^{n_1}} \prod_{k=1}^N\mathbf{1}_{[-1,1]}\left(\langle X^{(1)},\theta_1\rangle\right)\prod_{i=1}^{n_1}\mathbf{1}_{\left[-\frac{1}{2},\frac{1}{2}\right]}\left(\left\langle X^{(1)},e_i\right\rangle\right)\dd X^{(1)}
    \end{align*}
    This follows from Kanter's theorem (\textsc{Theorem} \ref{Kanter}). For the $\ell$-fold integral, by choosing $p=2,\dots,\ell$ and using Fubini's theorem, we have
    \begin{align*}
        &\mathbb{P}(\otimes_{j=1}^{\ell} X^{(j)}\in K) \\
        =&\int_{\mathbb{R}^{n_1}}\dots\int_{\mathbb{R}^{n_{\ell}}}\prod_{k=1}^N\mathbf{1}_{[-1,1]}\left(\langle X^{(p)},u^p\rangle\right)\prod_{i=1}^{n_1}{f_i^{(1)}}^*\left(\left\langle X^{(1)},e_i\right\rangle\right)\dd X^{(1)}\dots\prod_{i=1}^{n_{\ell}}{f_i^{(\ell)}}^*\left(\left\langle X^{(\ell)},e_i\right\rangle\right)\dd X^{(\ell)} \\
        =&\mathbb{P}(\otimes_{j=1}^{\ell} {X^{(j)}}^*\in K).
    \end{align*}
\end{proof}
By definition, stochastic dominance is transitive. Hence combining \textsc{Theorem} \ref{dominance1} and \ref{dominance2}, we have the corrollary.
\begin{corollary}\label{3.4}
    Let $X^{(j)}\in\mathbb{R}^{n_j}, 1\leq j\leq\ell$ be independent random vectors with independent coordinates whose densities have uniform norms bounded by 1, then
        $$\otimes_{j=1}^{\ell} X^{(j)}\prec\otimes_{j=1}^{\ell} {X^{(j)}}^*\prec\otimes_{j=1}^{\ell}\mathbf{1}_{{\left[-\frac{1}{2},\frac{1}{2}\right]}^{n_j}}.$$
\end{corollary}

\section{Anti-concentration of "log-concave" simple random tensors}
In this section we will prove \textsc{Theorem} \ref{1.4}. Suppose $X^{(j)}=\left(X_1^{(j)},\cdots,X_{n_j}^{(j)}\right), 1\leq j\leq\ell$ are random vectors in $\mathbb{R}^{n_j}$. Define the simple random tensor
    $$X:=X^{(1)}\otimes\cdots\otimes X^{(\ell)}=\left(X^{(1)}_{i_1}\cdots X^{(\ell)}_{i_{\ell}}\right)_{i_1\cdots i_{\ell}}.$$
Let $F$ be an $m$-dimensional subspace in $\mathbb{R}^{n_1\otimes\cdots\otimes n_{\ell}}$ and let $f^{(1)},\cdots,f^{(m)}$ be an orthonormal basis for $F$. Denote by $\Pi_FX^{(1)}\otimes\cdots\otimes X^{(\ell)}$ the orthogonal projection of $X^{(1)}\otimes\cdots\otimes X^{(\ell)}$ onto $F$. Then by definition we have
    $$\norm{\Pi_FX^{(1)}\otimes\cdots\otimes X^{(\ell)}}_2^2=\sum_{k=1}^m\left|\left\langle X^{(1)}\otimes\cdots\otimes X^{(\ell)},f^{(k)}\right\rangle\right|^2.$$
Here we apply the following version of Guedon's theorem, which is an application of \textsc{Theorem} \ref{guedonthm}.
\begin{lemma}\label{Guedon}
    Let $X\in\mathbb{R}^n$ be a log-concave random vector and let $\norm{\cdot}$ be a seminorm. Then for every $\varepsilon\in(0,1)$ we have 
    \begin{equation*}
        \mathbb{P}\left(\norm{X}\leq\varepsilon\mathbb{E}\norm{X}\right) \leq C\varepsilon.
    \end{equation*} 
\end{lemma}
We start the proof of \textsc{Theorem} \ref{1.4} for the case $\ell=2$. We would like to show that if $X^{(1)}\in\mathbb{R}^{n_1},X^{(2)}\in\mathbb{R}^{n_2}$ are independent isotropic log-concave random vectors and $f$ is in the orthonormal basis for subspace $F\in\mathbb{R}^{n_1\times n_2}$, then 
    $$\mathbb{P}\left(\left|\left\langle X^{(1)}\otimes X^{(2)},f\right\rangle\right|\leq\varepsilon\right)\leq C\varepsilon\log{\frac{1}{\varepsilon}}.$$ 
Note that for log-concave random vector $X$, $\mathbb{E}\norm{X}_2$ and $\sqrt{\mathbb{E}\norm{X}_2^2}$ are equivalent by Borell's lemma (see \textsc{Theorem} 1.5.7 in \cite{AGGbook1}). Set $\theta_{X^{(2)}}=\left(\sum_{i_2=1}^{n_2}X^{(2)}_{i_2}f_{i_1i_2}\right)_{i_1=1}^{n_1}$, which is linear image of $X^{(2)}$, therefore also log-concave. Note that
    $$\mathbb{E}\norm{\theta_{X^{(2)}}}_2^2=\sum_{i_1=1}^{n_1}\mathbb{E}\left(\sum_{i_2=1}^{n_2}X^{(2)}_{i_2}f_{i_1i_2}\right)^2=\sum_{i_1=1}^{n_1}\sum_{i_2=1}^{n_2}f_{i_1i_2}^2=1,$$
    $$\mathbb{E}\norm{\theta_{X^{(2)}}}_2\geq c_0,$$
where $0<c_0<1$ is some universal constant. Then for $0<\varepsilon<c_0$, we apply \textsc{Lemma} \ref{Guedon}, and we have
\begin{equation}\label{eq4}
    \mathbb{P}\left(\left|\left\langle\ X^{(1)},\theta\right\rangle\right|\leq\varepsilon\right) \leq C_1\varepsilon,
\end{equation}
\begin{equation}\label{eq5}
    \mathbb{P}\left(\norm{\theta_{X^{(2)}}}_2\leq\varepsilon\right) \leq C_1\varepsilon.
\end{equation}
Define a partition as follows 
    $$B_{\lambda}^{c}=\{\norm{\theta_{X^{(2)}}}_2<\lambda\},$$
    $$B_{\lambda,p}=\{\lambda2^p<\norm{\theta_{X^{(2)}}}_2<\lambda2^{p+1}\},\quad p=0,1,\cdots,N,$$
    $$B_{\lambda,N+1}=\{\norm{\theta_{X^{(2)}}}_2>\lambda2^{N+1}\}.$$
Let $c=\min\left\{c_0,e^{-1}\right\}$. Then for $0<\varepsilon<c^2$, 
\begin{align*}
    &\mathbb{P}\left(\sum_{i_1,i_2}X^{(1)}_{i_1}X^{(2)}_{i_2}f_{i_1i_2}\leq\varepsilon\right)\\
    =&\mathbb{P}\left(\left|\left\langle X^{(1)},\theta_{X^{(2)}}\right\rangle\right|\leq\varepsilon\right)\\
    =&\mathbb{E}_{X^{(2)}}\mathbb{P}_{X^{(1)}}\left(\left|\left\langle X^{(1)},\theta_{X^{(2)}}\right\rangle\right|\leq\varepsilon\right)\\
    =&\mathbb{E}_{X^{(2)}}\mathbf{1}_{B^c_{\lambda}}\mathbb{P}_{X^{(1)}}\left(\left|\left\langle X^{(1)},\theta_{X^{(2)}}\right\rangle\right|\leq\varepsilon\right)+\sum_{p=0}^{N}\mathbb{E}_{X^{(2)}}\mathbf{1}_{B_{\lambda,p}}\mathbb{P}_{X^{(1)}}\left(\left|\left\langle X^{(1)},\theta_{X^{(2)}}\right\rangle\right|\leq\varepsilon\right)\\
    +&\mathbb{E}_{X^{(2)}}\mathbf{1}_{B_{\lambda,N+1}}\mathbb{P}_{X^{(1)}}\left(\left|\left\langle X^{(1)},\theta_{X^{(2)}}\right\rangle\right|\leq\varepsilon\right).
\end{align*}

\noindent Here we estimate each term by repeatedly applying (\ref{eq4}) and (\ref{eq5}),
    $$\mathbb{E}_{X^{(2)}}\mathbf{1}_{B^c_{\lambda}}\mathbb{P}_{X^{(1)}}\left(\left|\left\langle X^{(1)},\theta_{X^{(2)}}\right\rangle\right|\leq\varepsilon\right)\leq\mathbb{P}\left(B_{\lambda}^c\right)=\mathbb{P}\left(\norm{\theta_{X^{(2)}}}_2<\lambda\right)\leq C_1\lambda,$$
and
\begin{align*}
    &\mathbb{E}_{X^{(2)}}\mathbf{1}_{B_{\lambda,p}}\mathbb{P}_{X^{(1)}}\left(\left|\left\langle X^{(1)},\theta_{X^{(2)}}\right\rangle\right|\leq\varepsilon\right)\\
    \leq&\mathbb{E}_{X^{(2)}}\mathbf{1}_{B_{\lambda,p}}\mathbb{P}_{X^{(1)}}\left(\left|\left\langle X^{(1)},\frac{\theta_{X^{(2)}}}{\norm{\theta_{X^{(2)}}}_2}\right\rangle\right|\leq\frac{\varepsilon}{\lambda2^p}\right)\\
    \leq&\frac{C_1\varepsilon}{\lambda2^p}\mathbb{P}\left(B_{\lambda,p}\right)\\
    \leq&\frac{C_1\varepsilon}{\lambda2^p}\mathbb{P}\left(\norm{\theta_{X^{(2)}}}_2\leq2^{p+1}\lambda\right)\\
    \leq&C_2\varepsilon,
\end{align*}
and
    $$\mathbb{E}_{X^{(2)}}\mathbf{1}_{B_{\lambda,N+1}}\mathbb{P}_{X^{(1)}}\left(\left|\left\langle X^{(1)},\theta_{X^{(2)}}\right\rangle\right|\leq\varepsilon\right)\leq\mathbb{E}_{X^{(2)}}\mathbb{P}_{X^{(1)}}\left(\left|\left\langle X^{(1)},\frac{\theta_{X^{(2)}}}{\norm{\theta_{X^{(2)}}}_2}\right\rangle\right|\leq\frac{\varepsilon}{\lambda2^{N+1}}\right)\leq \frac{\varepsilon}{\lambda2^{N+1}}.$$
Hence
    $$\mathbb{P}\left(\left|\langle X^{(1)}\otimes X^{(2)},f\rangle\right|\leq\varepsilon\right)\leq C_1\lambda+C_2N\varepsilon+\frac{\varepsilon}{\lambda2^{N+1}}.$$
Choose $\lambda=\frac{\varepsilon}{c}$ and $N=\log_2\frac{1}{\varepsilon}$, then
    $$\mathbb{P}\left(\left|\langle X^{(1)}\otimes X^{(2)},f\rangle\right|\leq\varepsilon\right)\leq \left(\frac{C_1}{c}+\frac{c}{2}\right)\varepsilon+C_2\varepsilon\log_2\frac{1}{\varepsilon}\leq C\varepsilon\log\frac{1}{\varepsilon}.$$

Now we consider the general case for $\ell\geq3$ by induction. Suppose it's true that for independent log-concave random vectors $X^{(j)}\in\mathbb{R}^{n_j}, 1\leq j\leq\ell-1$ whose coordinates are mean zero, uncorrelated and have variances bounded by $1$ and $f$ in the orthonormal basis for subspace $F\in\mathbb{R}^{n_1\times\cdots\times n_{\ell-1}}$, for $0<\varepsilon<c^{\ell-1}$ we have
    $$\mathbb{P}\left(\left|\langle X^{(1)}\otimes\cdots\otimes X^{(\ell-1)},f\rangle\right|\leq\varepsilon\right)\leq \frac{\varepsilon}{(\ell-2)!}\left(C\log\frac{1}{\varepsilon}\right)^{\ell-2}.$$
Then for independent log-concave random vectors $X^{(j)}\in\mathbb{R}^{n_j}, 1\leq i\leq\ell$ whose coordinates are mean zero, uncorrelated and have variances bounded by $1$ and $f$ in the orthonormal basis for subspace $F\in\mathbb{R}^{n_1\otimes\cdots\otimes n_{\ell}}$, for $0<\varepsilon<c^{\ell}$,
\begin{align*}
    &\mathbb{P}\left(\left|\left\langle X^{(1)}\otimes\cdots\otimes X^{(\ell)},f\right\rangle\right|\leq\varepsilon\right)\\
    =&\mathbb{P}\left(\left|\sum_{i_1\cdots i_{\ell}}X^{(1)}_{i_1}\cdots X^{(\ell)}_{i_{\ell}}f_{i_1\cdots i_{\ell}}\right|\leq\varepsilon\right)\\
    =&\mathbb{E}_{X^{(\ell)}}\mathbb{P}_{X^{(1)}\cdots X^{(\ell-1)}}\left(\left|\sum_{i_1\cdots i_{\ell-1}}X^{(1)}_{i_1}\cdots X^{(\ell-1)}_{i_{\ell-1}}\left(\sum_{i_{\ell}}X^{(\ell)}_{i_{\ell}}f_{i_1\cdots i_{\ell}}\right)\right|\leq\varepsilon\right)\\
    =&\mathbb{E}_{X^{(\ell)}}\mathbb{P}_{X^{(1)}\cdots X^{(\ell-1)}}\left(\left|\left\langle X^{(1)}\otimes\cdots\otimes X^{(\ell-1)},\theta_{X^{(\ell)}}\right\rangle\right|\leq\varepsilon\right)
\end{align*}
where 
    $$\theta_{X^{(\ell)}}:=\left(\sum_{i_{\ell}}X^{(\ell)}_{i_{\ell}}f_{i_1\cdots i_{\ell-1}i_{\ell}}\right)_{i_1\cdots i_{\ell-1}}=A_{\ell}X^{(\ell)}\in\mathbb{R}^{n_1\cdots n_{\ell-1}}$$
is a linear image of $X^{(\ell)}$ where
    $$A_{\ell}=\left(f_{i_1\cdots i_{\ell-1}i_{\ell}}\right)_{i_1\cdots i_{\ell-1},i_{\ell}}\in\mathbb{R}^{n_1\cdots n_{\ell-1}\times n_{\ell}}$$
is an $n_1\cdots n_{\ell-1}\times n_{\ell}$ matrix. Hence $\theta_{X^{(\ell)}}$ is also log-concave. Consider the Frobenius norm (Euclidean norm) $\norm{\cdot}_2$, then 
    $$\norm{\theta_{X^{(\ell)}}}_2=\left[\sum_{i_1\cdots i_{\ell-1}}\left(\sum_{i_{\ell}}X_{i_{\ell}}^{(\ell)}f_{i_1\cdots i_{\ell-1}i_{\ell}}\right)^2\right]^{1/2},$$
    $$\mathbb{E}\norm{\theta_{X^{(\ell)}}}_2^2=\norm{f_{i_1\cdots i_{\ell}}}_2^2=1,$$
    $$\mathbb{E}\norm{\theta_{X^{(\ell)}}}_2\geq c,$$
where $0<c<1$ is some universal constant. Then for $0<\varepsilon<c_0$, we have
\begin{equation*}
    \mathbb{P}\left(\norm{\theta_{X^{(\ell)}}}_2\leq\varepsilon\right) \leq C_1\varepsilon.
\end{equation*}
Let $c=\min\left\{c_0,e^{-1}\right\}$. Then for $0<\varepsilon<c^{\ell}$,
\begin{align*}
    &\mathbb{P}\left(\left|\left\langle X^{(1)}\otimes\cdots\otimes X^{(\ell)},f\right\rangle\right|\leq\varepsilon\right)\\
    =&\mathbb{E}_{X^{(\ell)}}\mathbb{P}_{X^{(1)}\cdots X^{(\ell-1)}}\left(\left|\left\langle X^{(1)}\otimes\cdots\otimes X^{(\ell-1)},\frac{\theta_{X^{(\ell)}}}{\norm{\theta_{X^{(\ell)}}}_2}\right\rangle\right|\leq\frac{\varepsilon}{\norm{\theta_{X^{(\ell)}}}_2}\right)\\
    =&\mathbb{E}_{X^{(\ell)}}\mathbf{1}_{\left\{\norm{\theta_{X^{(\ell)}}}_2\leq c^{1-\ell}\varepsilon\right\}}\mathbb{P}_{X^{(1)}\cdots X^{(\ell-1)}}\left(\left|\left\langle X^{(1)}\otimes\cdots\otimes X^{(\ell-1)},\frac{\theta_{X^{(\ell)}}}{\norm{\theta_{X^{(\ell)}}}_2}\right\rangle\right|\leq\frac{\varepsilon}{\norm{\theta_{X^{(\ell)}}}_2}\right)\\
    +&\mathbb{E}_{X^{(\ell)}}\mathbf{1}_{\left\{\norm{\theta_{X^{(\ell)}}}_2> c^{1-\ell}\varepsilon\right\}}\mathbb{P}_{X^{(1)}\cdots X^{(\ell-1)}}\left(\left|\left\langle X^{(1)}\otimes\cdots\otimes X^{(\ell-1)},\frac{\theta_{X^{(\ell)}}}{\norm{\theta_{X^{(\ell)}}}_2}\right\rangle\right|\leq\frac{\varepsilon}{\norm{\theta_{X^{(\ell)}}}_2}\right)\\
    \leq&\mathbb{P}_{X^{(\ell)}}\left(\norm{\theta_{X^{(\ell)}}}_2\leq c^{1-\ell}\varepsilon\right)+\mathbb{E}_{X^{(\ell)}}\mathbf{1}_{\left\{\norm{\theta_{X^{(\ell)}}}_2> c^{1-\ell}\varepsilon\right\}}\frac{\varepsilon}{\norm{\theta_{X^{(\ell)}}}_2(\ell-2)!}\left(C\log{\frac{\norm{\theta_{X^{(\ell)}}}_2}{\varepsilon}}\right)^{\ell-2}\\
    \leq&\frac{C_1}{c^{\ell}}\varepsilon+\mathbb{E}_{X^{(\ell)}}\mathbf{1}_{\left\{\norm{\theta_{X^{(\ell)}}}_2> c^{1-\ell}\varepsilon\right\}}\frac{\varepsilon}{\norm{\theta_{X^{(\ell)}}}_2(\ell-2)!}\left(C\log{\frac{\norm{\theta_{X^{(\ell)}}}_2}{\varepsilon}}\right)^{\ell-2}\\
    =&\frac{C_1}{c^{\ell}}\varepsilon+\frac{C^{\ell-2}}{(\ell-2)!}\mathbb{E}_{X^{(\ell)}}\mathbf{1}_{\left\{\norm{\theta_{X^{(\ell)}}}_2> c^{1-\ell}\varepsilon\right\}}\frac{\varepsilon}{\norm{\theta_{X^{(\ell)}}}_2}\left(\log{\frac{\norm{\theta_{X^{(\ell)}}}_2}{\varepsilon}}\right)^{\ell-2}
\end{align*}
It now suffices to show that 
    $$\mathbb{E}\mathbf{1}_{\left\{\norm{\theta_{X^{(\ell)}}}_2> c^{1-\ell}\varepsilon\right\}}\frac{\varepsilon}{\norm{\theta_{X^{(\ell)}}}_2}\left(\log{\frac{\norm{\theta_{X^{(\ell)}}}_2}{\varepsilon}}\right)^{\ell-2}\leq \frac{C}{\ell-1}\varepsilon\left(\log{\frac{1}{\varepsilon}}\right)^{\ell-1},$$
which is proved in \textsc{Lemma} \ref{B1}. This completes the proof of
    $$\mathbb{P}\left(\left|\left\langle X^{(1)}\otimes\cdots\otimes X^{(\ell)},f\right\rangle\right|\leq\varepsilon\right)\leq \frac{\varepsilon}{(\ell-1)!}\left(C\log\frac{1}{\varepsilon}\right)^{\ell-1}.$$

On one hand, recall that
    $$\norm{\Pi_FX^{(1)}\otimes\dots\otimes X^{(\ell)}}_2^2=\sum_{k=1}^m\left|\left\langle\otimes_{j=1}^{\ell}X^{(j)},f^{(k)}\right\rangle\right|^2.$$
Then by a union bound we have
    $$\mathbb{P}\left(\norm{\Pi_FX^{(1)}\otimes\dots\otimes X^{(\ell)}}_2\leq\varepsilon\sqrt{m}\right)\leq\frac{\varepsilon m}{(\ell-1)!}\left(C\log\frac{1}{\varepsilon}\right)^{\ell-1}.$$

On the other hand, let $0<q<1$, then
\begin{align*}
    \mathbb{E}\frac{1}{\norm{\Pi_F\otimes_{j=1}^{\ell}X^{(j)}}_2^q}=&\frac{1}{m^{\frac{q}{2}}}\mathbb{E}\frac{1}{\left(\frac{1}{m}\sum_{k=1}^m\left|\left\langle\otimes_{j=1}^{\ell}X^{(j)},f^{(k)}\right\rangle\right|^2\right)^{\frac{q}{2}}}\\
    \leq&\frac{1}{m^{\frac{q}{2}}}\mathbb{E}\frac{1}{\prod_{k=1}^m\left|\left\langle\otimes_{j=1}^{\ell}X^{(j)},f^{(k)}\right\rangle\right|^{\frac{q}{m}}}\\
    \leq&\frac{1}{m^{\frac{q}{2}}}\prod_{k=1}^m\left(\mathbb{E}\frac{1}{\left|\left\langle\otimes_{j=1}^{\ell}X^{(j)},f^{(k)}\right\rangle\right|^q}\right)^{\frac{1}{m}},
\end{align*}
where the first inequality follows from arithmetic and geometric mean inequality, and the second inequality follows from H\"older's inequality.
Note that for $1\leq k\leq m$, we apply \textsc{Lemma} \ref{C1} with $K=\frac{C^{\ell-1}}{(\ell-1)!}$, then for $1-\frac{1}{C}\leq q<1$, we have 
\begin{equation*}
    \mathbb{E}\frac{1}{\left|\left\langle\otimes_{j=1}^{\ell}X^{(j)},f^{(k)}\right\rangle\right|^q}\leq \frac{\left(K\ell^{\ell}\right)^q}{(1-q)^{q\left(\ell-1\right)+1}}.
\end{equation*}
By \textsc{Lemma} \ref{C2}, we take $0<\varepsilon<e^{-C\ell}$ and $q=1-\frac{\ell}{\ell+\log{\frac{1}{\varepsilon}}}$. Then by Markov's inequality,
    \begin{align*}
        \mathbb{P}\left(\norm{\Pi_F\otimes_{j=1}^{\ell}X^{(j)}}_2\leq\varepsilon\sqrt{m}\right)&=\mathbb{P}\left(\frac{1}{\norm{\Pi_F\otimes_{j=1}^{\ell}X^{(j)}}_2^q}\geq\frac{1}{\varepsilon^qm^{\frac{q}{2}}}\right)\\
        &\leq\varepsilon^q\frac{\left(K\ell^{\ell}\right)^q}{(1-q)^{q\left(\ell-1\right)+1}}\\
        &\leq K\left(3e\right)^{\ell}\varepsilon\left(\log{\frac{1}{\varepsilon}}\right)^{\ell}\\
        &=\frac{{C'}^{\ell}\varepsilon}{(\ell-1)!}\log{\frac{1}{\varepsilon}}\left(C\log{\frac{1}{\varepsilon}}\right)^{\ell-1}.
    \end{align*}

\section{Proof of \textsc{Theorem} \ref{1.5}}
We now consider generic subspace and prove \textsc{Theorem} \ref{1.5}. We are working under the assumption that the random tensor $\otimes_{j=2}^{\ell}X^{(j)}$ is concentrated around the sphere $n^{\frac{\ell-1}{2}}\mathbb{S}^{n^{\ell-1}-1}$. The following two propositions show that the assumption holds with high probability.
\begin{proposition}
    Let $X^{(1)},\cdots,X^{(\ell)}$ be independent isotropic log-concave random vectors in $\mathbb{R}^n$, then for $t\geq2$,
        $$\mathbb{P}\left(\norm{\otimes_{j=1}^{\ell}X^{(j)}}_2\geq t^{\ell}n^{\frac{\ell}{2}}\right)\leq e^{-ct\sqrt{n}\ell},$$
\end{proposition}
\begin{proof}
    By Markov's inequality and \textsc{Theorem} \ref{Paouris},
    \begin{align*}
        \mathbb{P}\left(\norm{\otimes_{j=1}^{\ell}X^{(j)}}_2\geq s^{\ell}C^{\ell}\left(\sqrt{n}+q\right)^{\ell}\right)\leq&\frac{\mathbb{E}\norm{\otimes_{j=1}^{\ell}X^{(j)}}_2^q}{s^{q\ell}C^{q\ell}\left(\sqrt{n}+q\right)^{q\ell}}\\
        =&\frac{\prod_{j=1}^{\ell}\mathbb{E}\norm{X^{(j)}}_2^q}{s^{q\ell}C^{q\ell}\left(\sqrt{n}+q\right)^{q\ell}}\\
        \leq&\frac{1}{s^{q\ell}}.
    \end{align*}
    Take $s=\frac{1}{e}$ and $q=\left(\frac{e}{C}t-1\right)\sqrt{n}$ for $t\geq2$, then
        $$\mathbb{P}\left(\norm{\otimes_{j=1}^{\ell}X^{(j)}}_2\geq \left(t\sqrt{n}\right)^{\ell}\right)\leq e^{-(t-1)\sqrt{n}\ell}\leq e^{-ct\sqrt{n}\ell}.$$
\end{proof}
\begin{proposition}\label{2-sided}
    Let $X^{(1)},\cdots,X^{(\ell)}$ be independent isotropic log-concave random vectors in $\mathbb{R}^n$ with Poincar\'e constant $C_P$, where $n$ is bounded from below and $\ell<\frac{2\sqrt{n}}{C_P}$ for some universal constant $C$.  Then for $t>0$,
        $$\mathbb{P}\left(\norm{\otimes_{j=1}^{\ell}X^{(j)}}_2\geq(1+t){n^{\frac{\ell}{2}}}\right)\leq \exp\left\{-C\min\left(\frac{n\left(\log{(1+t)}\right)^2}{C^2C_P^2\ell },\frac{\sqrt{n}\left(\log{(1+t)}\right)}{CC_P}\right)\right\},$$
    in particular, if $0<t<1$,
        $$\mathbb{P}\left(\norm{\otimes_{j=1}^{\ell}X^{(j)}}_2\geq(1+t){n^{\frac{\ell}{2}}}\right)\leq \exp\left\{-C\min\left(\frac{nt^2}{C^2C_P^2\ell },\frac{\sqrt{n}t}{CC_P}\right)\right\}.$$
    And for $1-\frac{1}{\left(1+\frac{C}{\sqrt{n}}\right)^{\ell}}\leq t<1$,
        $$\mathbb{P}\left(\norm{\otimes_{j=1}^{\ell}X^{(j)}}_2\leq(1-t){n^{\frac{\ell}{2}}}\right)\leq \exp\left\{-C\min\left(\frac{n\left(\log{\frac{\left(1-\frac{C_P}{2\sqrt{n}}\right)^{\ell}}{1-t}}\right)^2}{C^2C_P^2\ell },\frac{\sqrt{n}\log{\frac{\left(1-\frac{C_P}{2\sqrt{n}}\right)^{\ell}}{1-t}}}{CC_P}\right)\right\}.$$
\end{proposition}
\begin{proof}
    Define 
        $$Y_j=\log{\frac{\norm{X^{(j)}}_2}{\sqrt{n}}}.$$
    Note that
        $$\norm{\norm{X^{(j)}}_2}_{L_0}=\lim_{p\rightarrow0}\left(\mathbb{E}\norm{X}_2^p\right)^{\frac{1}{p}}=e^{\mathbb{E}\log{\norm{X}_2}},$$
    and by H\"older's inequality,
        $$\norm{\norm{X^{(j)}}_2}_{L_{-1}}\leq\norm{\norm{X^{(j)}}_2}_{L_0}\leq\norm{\norm{X^{(j)}}_2}_{L_1}\leq\sqrt{n}.$$
    Also by distribution formula
    \begin{align*}
        \mathbb{E}\frac{1}{\norm{X^{(j)}}_2}=&\int_0^{\infty}\mathbb{P}\left(\frac{1}{\norm{X^{(j)}_2}}\geq s\right)\dd s\\
        \leq&\int_0^{\frac{1}{\sqrt{n}-C_P}}1\dd s+\int_{\frac{1}{\sqrt{n}-C_P}}^{\infty}\mathbb{P}\left(\norm{X^{(j)}_2}\leq\frac{1}{s}\right)\dd s\\
        =&\frac{1}{\sqrt{n}-C_P}+\frac{1}{\sqrt{n}}\int_{0}^{1-\frac{C_P}{\sqrt{n}}}\frac{1}{t^2}\mathbb{P}\left(\norm{X^{(j)}_2}\leq t\sqrt{n}\right)\dd t\\
        \leq&\frac{1}{\sqrt{n}-C_P}+\frac{1}{\sqrt{n}}\int_{0}^{\frac{1}{2}}\frac{1}{t^2}t^{c\sqrt{n}}\dd t+\frac{4}{\sqrt{n}}\int_{\frac{1}{2}}^{1-\frac{C_P}{\sqrt{n}}}e^{-\frac{C'(1-t)\sqrt{n}}{C_P}}\dd t\\
        \leq&\frac{1}{\sqrt{n}-C_P}+\left.\frac{1}{\sqrt{n}\left(c\sqrt{n}-1\right)}t^{c\sqrt{n}-1}\right|_0^{\frac{1}{2}}+\left.\frac{4C_P}{\sqrt{n}\left(C'\sqrt{n}\right)}e^{-\frac{C'(1-t)\sqrt{n}}{C_P}}\right|_{\frac{1}{2}}^{1-\frac{C_P}{\sqrt{n}}}\\
        \leq&\frac{1}{\sqrt{n}-C_P}+\frac{1}{2^{c\sqrt{n}}n}+\frac{4C_Pe^{-C'}}{C'n}\\
        \leq&\frac{1}{\sqrt{n}-\frac{C_P}{2}}
    \end{align*}
    where we used \textsc{Theorem} \ref{PaourisSBP}. Then
        $$\sqrt{n}-\frac{C_P}{2}\leq\norm{\norm{X^{(j)}}_2}_{L_0}\leq\mathbb{E}\norm{X^{(j)}}_2\leq\sqrt{n},$$
    \begin{align*}
        \mathbb{E}\left[Y_j\right]=&\mathbb{E}\log{\norm{X^{(j)}}_2}-\log{\sqrt{n}}\\
        =&\log{\norm{\norm{X^{(j)}}_2}_{L_0}}-\log{\sqrt{n}}\\
        \in&\left[\log{\left(1-\frac{C_P}{2\sqrt{n}}\right)},0\right].
    \end{align*}
    Define
        $$\overline{Y_j}=Y_j-\mathbb{E}\left[Y_j\right],$$
    then for $s\geq1$,
    \begin{align*}
        \norm{\overline{Y_1}}_{L_s}^s\leq&\mathbb{E}\left|\log{\frac{\norm{X^{(1)}}_2}{\sqrt{n}}}-\mathbb{E}\left[Y_1\right]\right|^s\\
        =&\int_{0}^{\infty}su^{s-1}\mathbb{P}\left(\left|\log{\frac{\frac{\norm{X^{(1)}}_2}{\sqrt{n}}}{e^{\mathbb{E}\left[Y_1\right]}}}\right|\geq u\right)\dd u\\
        \leq&\int_{0}^{\infty}su^{s-1}\mathbb{P}\left(\frac{\norm{X^{(1)}}_2}{\sqrt{n}}\geq e^{\mathbb{E}\left[Y_1\right]}e^{u}\right)\dd u\\
        +&\int_0^{\infty}su^{s-1}\mathbb{P}\left(\frac{\norm{X^{(1)}}_2}{\sqrt{n}}\leq e^{\mathbb{E}\left[Y_1\right]}e^{-u}\right)\dd u\\
        \leq&\int_{0}^{\infty}su^{s-1}\mathbb{P}\left(\frac{\norm{X^{(1)}}_2}{\sqrt{n}}\geq \left(1-\frac{C_P}{2\sqrt{n}}\right)e^{u}\right)\dd u\\
        +&\int_0^{\infty}su^{s-1}\mathbb{P}\left(\frac{\norm{X^{(1)}}_2}{\sqrt{n}}\leq e^{-u}\right)\dd u\\
        :=&I_1+I_2.
    \end{align*}
    For the first integral, we have
    \begin{align*}
        I_1\leq&\int_0^{a}su^{s-1}\dd u+\int_{a}^{\infty}su^{s-1}\mathbb{P}\left(\frac{\norm{X^{(1)}}_2}{\sqrt{n}}\geq \left(1-\frac{C_P}{2\sqrt{n}}\right)e^{u}\right)\dd u,
    \end{align*}
    where $a\geq-\log{\left(1-\frac{C_P}{2\sqrt{n}}\right)}$. Consider $n$ sufficiently large so that $1-\frac{C_P}{2\sqrt{n}}\geq\frac{1}{2}$. This is guaranteed by Klartag's result. Take $a=\frac{3C_P}{\sqrt{n}}$, then for $u\geq a$, we have
    \begin{align*}
        \left(1-\frac{C_P}{2\sqrt{n}}\right)e^u\geq&\left(1-\frac{C_P}{2\sqrt{n}}\right)(1+u)\\
        =&1+\left(1-\frac{C_P}{2\sqrt{n}}\right)u-\frac{C_P}{2\sqrt{n}}\\
        \geq&1+\frac{1}{3}u+\left(\frac{1}{6}u-\frac{C_P}{2\sqrt{n}}\right)\\
        \geq&1+\frac{1}{3}u,
    \end{align*}
    \begin{align*}
        I_1\leq&\int_0^{\frac{3C_P}{\sqrt{n}}}su^{s-1}\dd u+\int_{\frac{3C_P}{\sqrt{n}}}^{\infty}su^{s-1}\mathbb{P}\left(\norm{X^{(1)}}_2\geq\left(1+\frac{1}{3}u\right)\sqrt{n}\right)\dd u\\
        \leq&\left(\frac{3C_P}{\sqrt{n}}\right)^s+\int_{\frac{3C_P}{\sqrt{n}}}^{\infty}su^{s-1}e^{-\frac{C'u\sqrt{n}}{3C_P}}\dd u\\
        \leq&\left(\frac{3C_P}{\sqrt{n}}\right)^s+\int_{C'}^{\infty}s\left(\frac{3C_Pv}{C'\sqrt{n}}\right)^{s-1}e^{-v}\frac{3C_P}{C'\sqrt{n}}\dd v\\
        \leq&\left(\frac{3C_P}{\sqrt{n}}\right)^s+\left(\frac{3C_P}{C'\sqrt{n}}\right)^ss\Gamma(s)\\
        \leq&\left(\frac{CC_Ps}{\sqrt{n}}\right)^s.
    \end{align*}
    For the second integral, notice that
        $$1-e^{-u}\geq\frac{u}{e}$$
    for $0\leq u\leq1$, then by Poincar\'e's inequality and \textsc{Theorem} \ref{PaourisSBP},
    \begin{align*}
        I_2\leq&\int_0^{1}su^{s-1}e^{-\frac{C'u\sqrt{n}}{eC_P}}\dd u+\int_{1}^{\infty}su^{s-1}\left(e^{-u}\right)^{c\sqrt{n}}\dd u\\  
        \leq&\left(\frac{eC_P}{C'\sqrt{n}}\right)^ss\Gamma(s)+\frac{1}{\left(c\sqrt{n}\right)^s}s\Gamma(s)\\
        =&\left(\frac{CC_Ps}{\sqrt{n}}\right)^s.
    \end{align*}
    Hence
        $$\norm{Y_j}_{L_s}\leq\frac{CC_P}{\sqrt{n}},$$
    By Berstein's inequality (see \cite{Vershynin_2018}), for $t\geq0$,
        $$\mathbb{P}\left(\left|\sum_{j=1}^{\ell}{\overline{Y_j}}\right|\geq t\right)\leq2\exp\left\{-C\min\left(\frac{t^2n}{C^2C_P^2\ell},\frac{t\sqrt{n}}{CC_P}\right)\right\}.$$
    Hence for $t>0$,
    \begin{align*}
        &\mathbb{P}\left(\prod_{j=1}^{\ell}\norm{X^{(j)}}_2\geq(1+t)n^{\frac{\ell}{2}}\right)\leq\mathbb{P}\left(\sum_{j=1}^{\ell}Y_j\geq\log{(1+t)}\right)\\
        \leq&\mathbb{P}\left(\sum_{j=1}^{\ell}\overline{Y_j}\geq\log{(1+t)}\right)\leq \exp\left\{-C\min\left(\frac{n\left(\log{(1+t)}\right)^2}{C^2C_P^2\ell },\frac{\sqrt{n}\left(\log{(1+t)}\right)}{CC_P}\right)\right\},
    \end{align*}
    and for $1-\frac{1}{\left(1+\frac{C}{\sqrt{n}}\right)^{\ell}}\leq t<1$,
    \begin{align*}
        &\mathbb{P}\left(\prod_{j=1}^{\ell}\norm{X^{(j)}}_2\leq(1-t)n^{\frac{\ell}{2}}\right)\\
        \leq&\mathbb{P}\left(\sum_{j=1}^{\ell}Y_j\leq\log{(1-t)}\right)\\
        \leq&\mathbb{P}\left(\sum_{j=1}^{\ell}\overline{Y_j}\leq\log{\left[\frac{1-t}{\left(1-\frac{C_P}{2\sqrt{n}}\right)^{\ell}}\right]}\right)\\
        \leq&\exp\left\{-C\min\left(\frac{n\left(\log{\frac{\left(1-\frac{C_P}{2\sqrt{n}}\right)^{\ell}}{1-t}}\right)^2}{C^2C_P^2\ell },\frac{\sqrt{n}\log{\frac{\left(1-\frac{C_P}{2\sqrt{n}}\right)^{\ell}}{1-t}}}{CC_P}\right)\right\}.
    \end{align*}
    This completes the proof.
\end{proof}
\begin{remark}\label{subgaussian}
    If $X^{(1)},\cdots,X^{(\ell)}$ are independent sub-gaussian random vectors with sub-gaussian norm $C_K$, then for $t>0$,
        $$\mathbb{P}\left(\norm{\otimes_{j=1}^{\ell}X^{(j)}}_2\geq(1+t){n^{\frac{\ell}{2}}}\right)\leq \exp\left\{-C\frac{n\left(\log{(1+t)}\right)^2}{{C'}^2C_K^2\ell}\right\}.$$
    and for $1-\frac{1}{\left(1+\frac{C}{\sqrt{n}}\right)^{\ell}}\leq t<1$,
        $$\mathbb{P}\left(\norm{\otimes_{j=1}^{\ell}X^{(j)}}_2\leq(1-t){n^{\frac{\ell}{2}}}\right)\leq \exp\left\{-C\frac{n\left(\log{\frac{\left(1-\frac{C_P}{2\sqrt{n}}\right)^{\ell}}{1-t}}\right)^2}{{C'}^2C_K^2\ell}\right\}.$$
    
\end{remark}

For log-concave random vectors we have the following small ball property.
\begin{proposition}\label{singular}
    Let $X\in\mathbb{R}^n$ be an isotropic log-concave random vector and let $A\in\mathbb{R}^{m\times n}$ such that $r=rank(A)=\min\{m,n\}$. Let $s_1\geq s_2\geq\cdots\geq s_r$ be the singular values of $A$ and let $A=U\Sigma V$ be its singular value decomposition, then for $0<\varepsilon<1$,
        $$\mathbb{P}\left(\norm{AX}_2\leq \varepsilon s_r\sqrt{r}\right)\leq \left(C\mathcal{L}_r\varepsilon\right)^r,$$
    where the isotropic constant
        $${\mathcal{L}}_{r}:=\sup_F\mathcal{L}_{\Pi_{F}}(\mu)=\sup_F\norm{f_{\Pi_FX}(x)}_{\infty}^{1/r},$$
    here $F$ is an $r$-dimensional subspace of $\mathbb{R}^n$ and $f_{\Pi_FX}(x)$ denotes the density of the marginal distribution $\Pi_FX$.
    In particular,
        $$\mathbb{P}\left(\norm{AX}_2\leq \varepsilon s_r\sqrt{r}\right)\leq \left(C\varepsilon\sqrt{\log{r}}\right)^r.$$
\end{proposition}
\begin{proof}
    Denote by $f_X$ the deinisty function of $X$.
    
    If $r=m$, then
        $$\mathbb{R}^n=C\left(A^T\right)\oplus\mathcal{N}(A)$$
    where $C\left(A^T\right)$ denotes the row space of $A$ and $\mathcal{N}(A)$ denotes the null space of $A$. For every $x\in\mathbb{R}^n$,
        $$x=\Pi_{C\left(A^T\right)}x+\Pi_{\mathcal{N}(A)}x.$$
    Then
        $$s_m\norm{\Pi_{C\left(A^T\right)}x}_2\leq\norm{Ax}_2=\norm{A\Pi_{C\left(A^T\right)}x}_2\leq s_1\norm{\Pi_{C\left(A^T\right)}x}_2,$$
    and
    \begin{align*}
        \mathbb{P}\left(\norm{AX}_2\leq \varepsilon s_m\sqrt{m}\right)\leq&\mathbb{P}\left(\norm{\Pi_{C\left(A^T\right)}X}_2\leq\varepsilon \sqrt{m}\right)\\
        \leq&\norm{f_{\Pi_{C\left(A^T\right)X}}(x)}_{\infty}\varepsilon^m\text{vol}\left(\mathbb{B}_2^m\right)\\
        \leq&\left(C\mathcal{L}_m\varepsilon\right)^m
    \end{align*}
    By Prékopa–Leindler inequality, $\Pi_{C\left(A^T\right)}X$ is an isotropic log-concave random vector in $\mathbb{R}^m$. Hence
        $$\mathcal{L}_m\leq C\sqrt{\log{m}}.$$
    Hence
        $$\mathbb{P}\left(\norm{AX}_2\leq \varepsilon s_m\sqrt{m}\right)\leq \left(C\varepsilon\sqrt{\log{m}}\right)^m.$$

    If $r=n$, then 
        $$s_n\norm{x}_2\leq\norm{Ax}_2\leq s_1\norm{x}_2.$$
    Hence
    \begin{align*}
        \mathbb{P}\left(\norm{AX}_2\leq \varepsilon s_n\sqrt{n}\right)\leq&\mathbb{P}\left(\norm{X}_2\leq\varepsilon\sqrt{n}\right)\leq\norm{f_X(x)}_{\infty}\varepsilon^n\text{vol}\left(\mathbb{B}_2^n\right)\\
        \leq&\left(\mathcal{L}_n\varepsilon\right)^n\leq\left(C\varepsilon\sqrt{\log{n}}\right)^n\\
    \end{align*}
\end{proof}
We also introduce the following lemma for special orthogonal groups.
\begin{lemma}\label{O(n)}
    (See Section 2.1 in \cite{elizabeth}) Let $U=\left(U_{ij}\right)_{ij}$ be distributed according to Haar measure in $\mathbb{O}(n)$, then
        $$\mathbb{E}\left[U_{ij}U_{i'j'}\right]=\frac{1}{n}\delta^{i',j'}_{i,j}.$$
\end{lemma}
\begin{lemma}\label{Grassmannian}
    Let $\mathcal{F}$ be an event in $\mathbb{O}(n)$. Fix $F_{0}$ a subspace in $G_{n,,m}$ and let $\mathcal{S}_{\mathcal{F}}:= \left\{ F \in G_{n,m} : F = U F_{0} , U \in\mathcal{F}\right\}$. Then
        $$\mathbb{P}_{\mathbb{O}\left(n\right)}\left(\mathcal{F}\right)=\mathbb{P}_{G_{n,m}}\left(\mathcal{S}_{\mathcal{F}}\right).$$
\end{lemma}

Now we are able to prove \textsc{Theorem} \ref{1.5}.
\begin{proof}
    Fix $t$ which will be chosen appropriately later on. We will consider several cases, and in each case, $t$ may be different. Denote
        $$E_t=\left\{X^{(2)},\cdots,X^{(\ell)}:\left|\frac{\prod_{j=2}^{\ell}\norm{X^{(j)}}_2}{n^{\frac{\ell-1}{2}}}-1\right|\leq t\right\},$$
    and 
        $$\overline{X}=X\mathbf{1}_{E_t},$$
        $$\mathcal{T}_{\varepsilon}=\left\{\norm{\Pi_F\otimes_{j=1}^{\ell}X^{(j)}}_2\leq\varepsilon\sqrt{m}\right\},$$
        $$\overline{\mathcal{T}}_{\varepsilon}=\left\{\norm{\Pi_F\otimes_{j=1}^{\ell}\overline{X}^{(j)}}_2\leq\varepsilon\sqrt{m}\right\}.$$
    Notice that
    \begin{align*}
        \mathbb{P}_X\left(\mathcal{T}_{\varepsilon}\right)=\mathbb{P}_X\left(\mathcal{T}_{\varepsilon}\cap E_t\right)+\mathbb{P}_X\left(\mathcal{T}_{\varepsilon}\cap E_t^C\right)\leq\mathbb{P}_{\overline{X}}\left(\overline{\mathcal{T}}_{\varepsilon}\right)\mathbb{P}_X\left(E_t\right)+\mathbb{P}_X\left(E_t^C\right).
    \end{align*}
    Hence it suffices to estimate $\mathbb{P}_{\overline{X}}\left(\mathcal{T}_{\varepsilon}\right)$. Let $U\in\mathbb{O}(n^{\ell})$ be a random special orthogonal matrix in Haar measure. For $F\in\mathbf{G}_{n^{\ell},m}$ in Haar measure, we can write its orthonormal basis as $f^{(k)}=Ue_k, 1\leq k\leq m$ which are the first $m$ columns of $U$. For any $A=\left(a_{i_1i_2\cdots i_{\ell}}\right)_{i_1i_2\cdots i_{\ell}}\in\mathbb{R}^{n^{\ell}}$, we denote by
    $$A^{1,2\cdots\ell}=\left(a_{i_1i_2\cdots i_{\ell}}\right)_{i_1,i_2\cdots i_{\ell}}\in\mathbf{Mat}_{n,n^{\ell-1}}\left(\mathbb{R}\right)$$
    an $n\times n^{\ell-1}$ matrix. And we consider a tensor as a flattened vector when there is no confusion. Define 
    \begin{align*}
        \theta_{X^{(2)},\cdots,X^{(\ell)}}\left(U\right)=&\left(\left(Ue_1\right)^{1,2\cdots \ell}\otimes_{j=2}^{\ell}X^{(j)},\cdots,\left(Ue_m\right)^{1,2\cdots \ell}\otimes_{j=2}^{\ell}X^{(j)}\right)^T\\
        =&\left({f^{(1)}}^{1,2\cdots \ell}\otimes_{j=2}^{\ell}X^{(j)},\cdots,{f^{(m)}}^{1,2\cdots \ell}\otimes_{j=2}^{\ell}X^{(j)}\right)^T.
    \end{align*}
    Then
    \begin{align*}
        \norm{\Pi_F\otimes_{j=1}^{\ell}X^{(j)}}_2^2=&\sum_{k=1}^m\left|\left\langle f^{(k)},\otimes_{j=1}^{\ell} X^{(j)}\right\rangle\right|^2=\sum_{k=1}^m\left|\left\langle {f^{(k)}}^{1,2\cdots \ell}\otimes_{j=2}^{\ell}X^{(j)},X^{(1)}\right\rangle\right|^2\\
        =&\norm{\theta_{X^{(2)},\cdots,X^{(\ell)}}(U)X^{(1)}}_2^2,
    \end{align*}
    and
        $$\norm{\Pi_F\otimes_{j=1}^{\ell}\overline{X}^{(j)}}_2^2=\norm{\theta_{\overline{X}^{(2)},\cdots,\overline{X}^{(\ell)}}(U)\overline{X}^{(1)}}_2^2.$$
    Notice that $f^{(k)}$ is the $k$-th column of $U$, then by \textsc{Lemma} \ref{O(n)}
        $$\mathbb{E}_U\left[f_{i_1\cdots i_{\ell}}^{(k)}f_{i'_1\cdots i'_{\ell}}^{(k')}\right]=\frac{1}{n^{\ell}}\delta_{k,i_1\cdots,i_{\ell}}^{k',i'_1,\cdots,i'_{\ell}},$$
    In the rest of the proof, we sometimes write $\theta=\theta_{\overline{X}^{(2)},\dots,\overline{X}^{(\ell)}}\left(U\right)$ if there is no confusion.
    
    \textbf{Case 1}: Suppose $m\leq C_1n$, where $0<C_1<1$ is a constant that is determined later. Then for any $\phi\in\mathbb{S}^{m-1}$,
        $$\norm{\theta_{\overline{X}^{(2)},\dots,\overline{X}^{(\ell)}}\left(U\right)^T\phi}_2^2=\sum_{i_1=1}^n\left(\sum_{k=1}^m\phi_k\sum_{i_2\cdots i_{\ell}}f_{i_1\cdots i_{\ell}}^{(k)}\overline{X}_{i_2}^{(2)}\cdots \overline{X}_{i_{\ell}}^{(\ell)}\right)^2.$$
    Fix $X^{(2)},\cdots,X^{(\ell)}$ and define $\mathcal{A}_{\phi}(U)=\norm{\theta_{\overline{X}^{(2)},\dots,\overline{X}^{(\ell)}}\left(U\right)^T\phi}_2$. Then for $U,U'\in\mathbb{O}(n^{\ell})$,
    \begin{align*}
        \left|\mathcal{A}_{\phi}(U)-\mathcal{A}_{\phi}(U')\right|=&\left|\norm{\sum_{k=1}^m\phi_k\left(Ue_k\right)^{1,2\cdots \ell}\otimes_{j=2}^{\ell}\overline{X}^{(j)}}_2-\norm{\sum_{k=1}^m\phi_k\left(U'e_k\right)^{1,2\cdots \ell}\otimes_{j=2}^{\ell}\overline{X}^{(j)}}_2\right|\\
        \leq&\norm{\left(\sum_{k=1}^m\phi_k\left[\left(U-U'\right)e_k\right]^{1,2\cdots \ell}\right)\otimes_{j=2}^{\ell}\overline{X}^{(j)}}_2
    \end{align*}
    Notice that $\sum_{k=1}^m\phi_k\left[\left(U-U'\right)e_k\right]^{1,2\cdots \ell}\in\mathbf{Mat}_{n,n^{\ell-1}}$ and $\otimes_{j=2}^{\ell}\overline{X}^{(j)}\in\mathbb{R}^{\ell-1}$, then
    \begin{align*}
        \norm{\left(\sum_{k=1}^m\phi_k\left[\left(U-U'\right)e_k\right]^{1,2\cdots \ell}\right)\otimes_{j=2}^{\ell}\overline{X}^{(j)}}_2\leq&\norm{\sum_{k=1}^m\phi_k\left[\left(U-U'\right)e_k\right]^{1,2\cdots \ell}}_{HS}\cdot\prod_{j=2}^{\ell}\norm{\overline{X}^{(j)}}_2\\
        =&\norm{\sum_{k=1}^m\phi_k\left(U-U'\right)e_k}_2\cdot\prod_{j=2}^{\ell}\norm{\overline{X}^{(j)}}_2\\
        \leq&(1+t)n^{\frac{\ell-1}{2}}\norm{U-U'}_{HS}.
    \end{align*}
    Hence $\mathcal{A}_{\phi}$ has Lipschitz constant less than or equal to $(1+t)n^{\frac{\ell-1}{2}}$. And
        $$\mathbb{E}_U\left[\mathcal{A}_{\phi}(U)^2\right]=\sum_{i_1=1}^n\sum_{i_2,\cdots,i_{\ell}}\frac{1}{n^{\ell}}{\overline{X}_{i_2}^{(2)}}^2\cdots {\overline{X}_{i_{\ell}}^{(\ell)}}^2=\frac{\prod_{j=2}^{\ell}\norm{\overline{X}_{i_j}^{(j)}}_2^2}{n^{\ell-1}}.$$
    Let $\mathcal{N}_1$ be a $\delta$-net on $\mathbb{R}^m$, $\forall\phi\in\mathbb{S}^{m-1}$, $\exists\phi'\in\mathcal{N}_1$, such that
        $$\norm{\phi-\phi'}\leq\delta.$$
    Recall $U\in\mathbb{O}\left(n^{\ell}\right)$. For every $\phi\in\mathcal{N}_1$, by Minkowski's inequality we have
    \begin{align*}
        &\left(\mathbb{E}_{\overline{X},U}\left|\mathcal{A}_{\phi}(U)-1\right|^q\right)^{\frac{1}{q}}\\
        \leq&\left(\mathbb{E}_{\overline{X},U}\left|\mathcal{A}_{\phi}(U)-\frac{\prod_{j=2}^{\ell}\norm{\overline{X}^{(j)}}_2}{n^{\frac{\ell-1}{2}}}\right|^q\right)^{\frac{1}{q}}+\left(\mathbb{E}_{\overline{X},U}\left|\frac{\prod_{j=2}^{\ell}\norm{\overline{X}^{(j)}}_2}{n^{\frac{\ell-1}{2}}}-1\right|^q\right)^{\frac{1}{q}}.
    \end{align*}
    For the first part, we apply \textsc{Corollary} \ref{moment} for the $1$-Lipschitz function $\frac{\mathcal{A}_{\phi}}{(1+t)n^{\frac{\ell-1}{2}}}$
    \begin{align*}
        \left(\mathbb{E}_{\overline{X},U}\left|\mathcal{A}_{\phi}(U)-\frac{\prod_{j=2}^{\ell}\norm{\overline{X}^{(j)}}_2}{n^{\frac{\ell-1}{2}}}\right|^q\right)^{\frac{1}{q}}\leq\left(\frac{Cq}{n^{\ell}}\right)^{\frac{1}{2}}(1+t)n^{\frac{\ell-1}{2}}\leq C(1+t)\sqrt{\frac{q}{n}}.
    \end{align*}
    For the second part, recall the definition of $\overline{X}$,
    \begin{align*}
        \left(\mathbb{E}_{\overline{X},U}\left|\frac{\prod_{j=2}^{\ell}\norm{\overline{X}^{(j)}}_2}{n^{\frac{\ell-1}{2}}}-1\right|^q\right)^{\frac{1}{q}}=t.
    \end{align*}
    Choose $q=\frac{(\eta-1)^2t^2}{C^2(1+t)^2}n$, where $\eta>1$. Then
        $$\mathbb{E}_U\mathbb{E}_{\overline{X}}\left|\mathcal{A}_{\phi}(U)-1\right|^{\frac{(\eta-1)^2t^2}{C^2(1+t)^2}n}=\mathbb{E}_{\overline{X},U}\left|\mathcal{A}_{\phi}(U)-1\right|^{\frac{(\eta-1)^2t^2}{C^2(1+t)^2}n}\leq\left(\eta t\right)^{\frac{(\eta-1)^2t^2}{C^2(1+t)^2}n}.$$
    Then by Markov's inequality, for $1<\sigma_1<\frac{e}{e-1}$,
        $$\mathbb{P}_U\left(\mathbb{E}_{\overline{X}}\left|\mathcal{A}_{\phi}(U)-1\right|^{\frac{(\eta-1)^2t^2}{C^2(1+t)^2}n}\geq \left(\sigma_1\eta t\right)^{\frac{(\eta-1)^2t^2}{C^2(1+t)^2}n}\right)\leq e^{-\log{(\sigma_1)}\frac{(\eta-1)^2t^2}{C^2(1+t)^2}n}\leq e^{-\log{(\sigma_1)}\frac{(\eta-1)^2t^2}{4C^2}n}.$$
    Denote by $\mathcal{F}$ the event in $\mathbb{O}\left(n^{\ell}\right)$ that 
        $$\mathbb{E}_{\overline{X}}\left|\mathcal{A}_{\phi}(U)-1\right|^{\frac{(\eta-1)^2t^2}{C^2(1+t)^2}n}\leq(\sigma_1\eta t)^{\frac{(\eta-1)^2t^2}{C^2(1+t)^2}n}$$
    for every $\phi\in\mathcal{N}_1$, then
        $$\mathbb{P}_U\left(\mathcal{F}\right)\geq1-\left(\frac{3}{\delta}\right)^me^{-\log{(\sigma_1)}\frac{(\eta-1)^2t^2}{4C^2}n}.$$
    We now fix a $U$ in the event $\mathcal{F}$. For every $\phi\in\mathcal{N}_1$, and for $1<\sigma_2<\frac{e}{\sigma_1(e-1)}$,
        $$\mathbb{P}_{\overline{X}}\left(\left\{\left|\mathcal{A}_{\phi}(U)-1\right|\geq \sigma_1\sigma_2\eta t\right\}\right)\leq e^{-\log{(\sigma_2)}\frac{(\eta-1)^2t^2}{4C^2}n}.$$
    Note that for $0<\delta<1$,
    \begin{align*}
        s_1\left(\theta^T\right)=&\max_{\phi\in\mathbb{S}^{m-1}}\norm{\theta^T\phi}_2\\
        \leq&\norm{\theta^T\phi'}_2+\max_{\phi\in\mathbb{S}^{m-1}}\{\theta^T\left(\phi-\phi'\right)\}\\
        \leq&\norm{\theta^T\phi'}_2+\delta\, s_1\left(\theta^T\right),
    \end{align*}
        $$s_1\left(\theta^T\right)\leq\frac{\norm{\theta^T\phi'}_2}{1-\delta}.$$
    Hence
    \begin{align*}
        \mathbb{P}_{\overline{X}}\left(s_1\left(\theta^T\right)\geq1+\sigma_1\sigma_2\eta t\right)\leq&\mathbb{P}_{\overline{X}}\left(\exists\phi'\in\mathcal{N}_1,\norm{\theta^T\phi'}_2\geq\left(1-\delta\right)\left(1+\sigma_1\sigma_2\eta t\right)\right)\\
        \leq&\left(\frac{3}{\delta}\right)^m\exp\left\{-\log{(\sigma_2)}\frac{(\eta-1)^2\left(\sigma_1\sigma_2\eta t-\delta\sigma_1\sigma_2\eta t-\delta\right)^2}{4C^2\sigma_1^2\sigma_2^2\eta^2}n\right\}.
    \end{align*}
    On the other hand, $\forall\phi\in\mathbb{S}^{m-1}$, $\exists\phi'\in\mathcal{N}_1$, such that
        $$s_n=\min_{\phi\in\mathbb{S}^{m-1}}\norm{\theta^T\phi}_2\geq\norm{\theta^T\phi'}-\delta \,s_1\left(\theta^T\right).$$
    Hence
    \begin{align*}
        &\mathbb{P}_{\overline{X}}\left(s_m\left(\theta^T\right)\leq1-\sigma_1\sigma_2\eta t\right)\\
        \leq&\mathbb{P}_{\overline{X}}\left(\exists\phi'\in\mathcal{N}_1,\norm{\theta^T\phi'}_2\leq1-\sigma_1\sigma_2\eta t+\delta\left(1+\sigma_1\sigma_2\eta t\right)\right)+\mathbb{P}_{\overline{X}}\left(s_1\left(\theta^T\right)\geq1+\sigma_1\sigma_2\eta t\right)\\
        \leq&2\left(\frac{3}{\delta}\right)^m\exp\left\{-\log{(\sigma_2)}\frac{(\eta-1)^2\left(\sigma_1\sigma_2\eta t-\delta\sigma_1\sigma_2\eta t-\delta\right)^2}{4C^2\sigma_1^2\sigma_2^2\eta^2}n\right\}.
    \end{align*}
    Pick $\delta=\frac{\sigma_1\sigma_2\eta t}{4}$ so that $\sigma_1\sigma_2\eta t-\delta \sigma_1\sigma_2\eta t-\delta\geq\frac{\sigma_1\sigma_2\eta t}{2}$, then
        $$\mathbb{P}_{\overline{X}}\left(s_m\left(\theta^T\right)\leq1-\sigma_1\sigma_2\eta t\right)\leq2\exp\left\{m\log{\left(\frac{12}{\sigma_1\sigma_2\eta t}\right)}-\log{(\sigma_2)}\frac{(\eta-1)^2t^2}{16C^2}n\right\}.$$
    Fix $\sigma_1=\sigma_2=\eta=\left(\frac{e}{e-1}\right)^{\frac{1}{12}}$ and fix $t=\left(\frac{e-1}{e}\right)^{\frac{1}{2}}$. Recall that $m\leq C_1n$, then we take
        $$C_1=\frac{\left[\left(\frac{e}{e-1}\right)^{\frac{1}{12}}-1\right]^2\frac{e-1}{e}\log{\frac{e}{e-1}}}{384C^2\log{12\left(\frac{e}{e-1}\right)^{\frac{1}{4}}}},$$
    then
        $$\mathbb{P}_{\overline{X}}\left(s_m\left(\theta^T\right)\leq1-\left(\frac{e-1}{e}\right)^{\frac{1}{4}}\right)\leq2\exp\left\{-\frac{\left[\left(\frac{e}{e-1}\right)^{\frac{1}{12}}-1\right]^2\frac{e-1}{e}\log{\frac{e}{e-1}}}{384C^2}n\right\}.$$
    By \textsc{Lemma} \ref{Grassmannian}, the event $\mathcal{F}$ in $\mathbb{O}\left(n^{\ell}\right)$ determines a subset $S_{\mathcal{F}}\subset\mathbf{G}_{n^{\ell},m}$ with measure at least $1-e^{-cn}$, for every $F\in S_{\mathcal{F}}$,
    \begin{align*}
        \mathbb{P}_{\overline{X}}\left(\overline{\mathcal{T}}_{\varepsilon}\right)=&\mathbb{P}_{\overline{X}}\left(\norm{\Pi_F\otimes_{j=1}^{\ell}\overline{X}^{(j)}}_2\leq \varepsilon\sqrt{m}\right)\\
        \leq&\mathbb{P}_{\overline{X}}\left(\norm{\Pi_F\otimes_{j=1}^{\ell}\overline{X}^{(j)}}_2\leq\frac{\varepsilon}{1-\left(\frac{e-1}{e}\right)^{\frac{1}{4}}} s_m\sqrt{m}\right)+\mathbb{P}_{\overline{X}}\left(s_m\leq 1-\left(\frac{e-1}{e}\right)^{\frac{1}{4}}\right)\\
        \leq&\left(C\varepsilon\mathcal{L}_m\right)^m+2e^{-cn},
    \end{align*}
    and by \textsc{Proposition} \ref{2-sided},
    \begin{align*}
        \mathbb{P}_X\left(\mathcal{T}_{\varepsilon}\right)\leq\mathbb{P}_{\overline{X}}\left(\overline{\mathcal{T}}_{\varepsilon}\right)\mathbb{P}_X\left(E_t\right)+\mathbb{P}_X\left(E_t^C\right)\leq\left(C\varepsilon\mathcal{L}_m\right)^m+e^{-\frac{c\sqrt{n}}{C_P}}.
    \end{align*}
    
    \textbf{Case 2}: Suppose $m\geq C_2n$, where $C_2>1$ is a constant that is determined later. Then for any $\psi\in\mathbb{S}^{n-1}$,
        $$\norm{\theta_{\overline{X}^{(2)},\dots,\overline{X}^{(\ell)}}\psi}_2^2=\sum_{k=1}^m\left|\left\langle {f^{(k)}}^{1,2\cdots \ell}\otimes_{j=2}^{\ell}\overline{X}^{(j)},\psi\right\rangle\right|^2=\sum_{k=1}^m\left(\sum_{i_1=1}\psi_{i_1}\sum_{i_2,\cdots,i_{\ell}}f^{(k)}_{i_1\cdots i_{\ell}}\overline{X}^{(2)}_{i_2}\cdots \overline{X}^{(\ell)}_{i_{\ell}}\right)^2,$$
    and
        $$\mathbb{E}_{\overline{X},U}\norm{\theta_{\overline{X}^{(2)},\dots,\overline{X}^{(\ell)}}\left(U\right)\psi}_2^2=\sum_{k=1}^m\sum_{i_1=1}^n\psi_{i_1}^2\sum_{i_2,\cdots,i_{\ell}}\frac{1}{n^{\ell}}=\frac{m}{n}.$$
     Define $\mathcal{B}_{\psi}(U)=\norm{\theta_{\overline{X}^{(2)},\dots,\overline{X}^{(\ell)}}\left(U\right)\psi}_2$, then for $U,U'\in\mathbb{O}(n^{\ell})$,
    \begin{align*}
        \left|\mathcal{B}_{\psi}(U)-\mathcal{B}_{\psi}(U')\right|=&\left|\norm{\theta(U)\psi}_2-\norm{\theta\left(U'\right)\psi}_2\right|\\
        \leq&\norm{\left(\theta(U)-\theta\left(U'\right)\right)\psi}_2\\
        =&\norm{\theta(U)-\theta\left(U'\right)}_2\norm{\psi}_2\\
        =&\sqrt{\sum_{k=1}^m\norm{\left[\left(U-U'\right)e_k\right]^{1,2...\ell}\otimes_{j=2}^{\ell}\overline{X}^{(j)}}}_2\\
        \leq&\norm{U-U'}_2\norm{\otimes_{j=2}^{\ell}\overline{X}^{(j)}}_2\\
        \leq&(1+t)n^{\frac{\ell-1}{2}}\norm{U-U'}_{HS}.
    \end{align*}
    Hence $\mathcal{B}$ has Lipschitz constant less than or equal to $(1+t)n^{\frac{\ell-1}{2}}$. And
        $$\mathbb{E}_U\left[{\mathcal{B}_{\psi}(U)}^2\right]=\sum_{k=1}^m\sum_{i_2,\cdots,i_{\ell}}\frac{1}{n^{\ell}}{\overline{X}_{i_2}^{(2)}}^2\cdots{\overline{X}_{i_{\ell}}^{(\ell)}}^2=\frac{m\prod_{j=2}^{\ell}\norm{\overline{X}^{(j)}}_2^2}{n^{\ell}}$$
    Let $\mathcal{N}_2$ be a $\delta$-net on $\mathbb{R}^n$, $\forall\psi\in\mathbb{S}^{n-1}$, $\exists\psi'\in\mathcal{N}_2$, such that
        $$\norm{\psi-\psi'}\leq\delta.$$
    Recall $U\in\mathbb{O}\left(n^{\ell}\right)$. For every $\psi\in\mathcal{N}_2$, by Minkowski's inequality we have
    \begin{align*}
        &\left(\mathbb{E}_{\overline{X},U}\left|\mathcal{B}_{\psi}(U)-\sqrt{\frac{m}{n}}\right|^q\right)^{\frac{1}{q}}\\
        \leq&\left(\mathbb{E}_{\overline{X},U}\left|\mathcal{B}_{\psi}(U)-\frac{\sqrt{m}\prod_{j=2}^{\ell}\norm{\overline{X}^{(j)}}_2}{n^{\frac{\ell}{2}}}\right|^q\right)^{\frac{1}{q}}+\left(\mathbb{E}_{\overline{X},U}\left|\frac{\sqrt{m}\prod_{j=2}^{\ell}\norm{\overline{X}^{(j)}}_2}{n^{\frac{\ell}{2}}}-\sqrt{\frac{m}{n}}\right|^q\right)^{\frac{1}{q}}.
    \end{align*}
    For the first part, we apply \textsc{Corollary} \ref{moment} for the $1$-Lipschitz function $\frac{\mathcal{B}_{\psi}}{(1+t)n^{\frac{\ell-1}{2}}}$
    \begin{align*}
        \left(\mathbb{E}_{\overline{X},U}\left|\mathcal{B}_{\psi}(U)-\frac{\sqrt{m}\prod_{j=2}^{\ell}\norm{\overline{X}^{(j)}}_2}{n^{\frac{\ell}{2}}}\right|^q\right)^{\frac{1}{q}}\leq\left(\frac{Cq}{n^{\ell}}\right)^{\frac{1}{2}}(1+t)n^{\frac{\ell-1}{2}}\leq C(1+t)\sqrt{\frac{q}{n}}.
    \end{align*}
    For the second part, recall the definition of $\overline{X}$,
    \begin{align*}
        \left(\mathbb{E}_{\overline{X},U}\left|\frac{\sqrt{m}\prod_{j=2}^{\ell}\norm{\overline{X}^{(j)}}_2}{n^{\frac{\ell}{2}}}-\sqrt{\frac{m}{n}}\right|^q\right)^{\frac{1}{q}}=t\sqrt{\frac{m}{n}}.
    \end{align*}

    Choose $q=\frac{(\eta-1)^2t^2}{C^2(1+t)^2}m$, where $\eta>1$. Then
        $$\mathbb{E}_U\mathbb{E}_{\overline{X}}\left|\mathcal{B}_{\psi}(U)-\sqrt{\frac{m}{n}}\right|^{\frac{(\eta-1)^2t^2}{C^2(1+t)^2}m}=\mathbb{E}_{\overline{X},U}\left|\mathcal{B}_{\psi}(U)-\sqrt{\frac{m}{n}}\right|^{\frac{(\eta-1)^2t^2}{C^2(1+t)^2}m}\leq\left(\eta t\sqrt{\frac{m}{n}}\right)^{\frac{(\eta-1)^2t^2}{C^2(1+t)^2}m}.$$
    Then by Markov's inequality, for $1<\sigma_1<\frac{e}{e-1}$,
        $$\mathbb{P}_U\left(\mathbb{E}_{\overline{X}}\left|\mathcal{B}_{\psi}(U)-\sqrt{\frac{m}{n}}\right|^{\frac{(\eta-1)^2t^2}{C^2(1+t)^2}m}\geq \left(\sigma_1\eta t\sqrt{\frac{m}{n}}\right)^{\frac{(\eta-1)^2t^2}{C^2(1+t)^2}m}\right)\leq e^{-\log{(\sigma_1)}\frac{(\eta-1)^2t^2}{C^2(1+t)^2}m}\leq e^{-\log{(\sigma_1)}\frac{(\eta-1)^2t^2}{4C^2}m}.$$
    Denote by $\mathcal{F}$ the event in $\mathbb{O}\left(n^{\ell}\right)$ that 
        $$\mathbb{E}_{\overline{X}}\left|\mathcal{B}_{\psi}(U)-\sqrt{\frac{m}{n}}\right|^{\frac{(\eta-1)^2t^2}{C^2(1+t)^2}m}\leq\left(\sigma_1\eta t\sqrt{\frac{m}{n}}\right)^{\frac{(\eta-1)^2t^2}{C^2(1+t)^2}m}$$
    for every $\psi\in\mathcal{N}_2$, then
        $$\mathbb{P}_U\left(\mathcal{F}\right)\geq1-\left(\frac{3}{\delta}\right)^ne^{-\log{(\sigma_1)}\frac{(\eta-1)^2t^2}{4C^2}m}.$$
    We now fix a $U$ in the event $\mathcal{F}$. For every $\psi\in\mathcal{N}_2$, and for $1<\sigma_2<\frac{e}{\sigma_1(e-1)}$,
        $$\mathbb{P}_{\overline{X}}\left(\left\{\left|\mathcal{B}_{\psi}(U)-\sqrt{\frac{m}{n}}\right|\geq \sigma_1\sigma_2\eta t\sqrt{\frac{m}{n}}\right\}\right)\leq e^{-\log{(\sigma_2)}\frac{(\eta-1)^2t^2}{4C^2}m}.$$
    Note that for $0<\delta<1$,
    \begin{align*}
        s_1\left(\theta\right)=&\max_{\psi\in\mathbb{S}^{n-1}}\norm{\theta\psi}_2\\
        \leq&\norm{\theta\psi'}_2+\max_{\psi\in\mathbb{S}^{n-1}}\{\theta\left(\psi-\psi'\right)\}\\
        \leq&\norm{\theta\psi'}_2+\delta\, s_1\left(\theta\right),
    \end{align*}
        $$s_1\left(\theta\right)\leq\frac{\norm{\theta\psi'}_2}{1-\delta}.$$
    Hence
    \begin{align*}
        \mathbb{P}_{\overline{X}}\left(s_1\left(\theta\right)\geq\left(1+\sigma_1\sigma_2\eta t\right)\sqrt{\frac{m}{n}}\right)\leq&\mathbb{P}_{\overline{X}}\left(\exists\psi'\in\mathcal{N}_2,\norm{\theta\psi'}_2\geq\left(1-\delta\right)\left(1+\sigma_1\sigma_2\eta t\right)\sqrt{\frac{m}{n}}\right)\\
        \leq&\left(\frac{3}{\delta}\right)^n\exp\left\{-\log{(\sigma_2)}\frac{(\eta-1)^2\left(\sigma_1\sigma_2\eta t-\delta\sigma_1\sigma_2\eta t-\delta\right)^2}{4C^2\sigma_1^2\sigma_2^2\eta^2}m\right\}.
    \end{align*}
    On the other hand, $\forall\psi\in\mathbb{S}^{n-1}$, $\exists\psi'\in\mathcal{N}_2$, such that
        $$s_n=\min_{\psi\in\mathbb{S}^{n-1}}\norm{\theta\psi}_2\geq\norm{\theta\psi'}-\delta \,s_1\left(\theta\right).$$
    Hence
    \begin{align*}
        &\mathbb{P}_{\overline{X}}\left(s_n\left(\theta\right)\leq\left(1-\sigma_1\sigma_2\eta t\right)\sqrt{\frac{m}{n}}\right)\\
        \leq&\mathbb{P}_{\overline{X}}\left(\exists\psi'\in\mathcal{N}_2,\norm{\theta\psi'}_2\leq\left[1-\sigma_1\sigma_2\eta t+\delta\left(1+\sigma_1\sigma_2\eta t\right)\right]\sqrt{\frac{m}{n}}\right)+\mathbb{P}_{\overline{X}}\left(s_1\left(\theta\right)\geq\left(1+\sigma_1\sigma_2\eta t\right)\sqrt{\frac{m}{n}}\right)\\
        \leq&2\left(\frac{3}{\delta}\right)^n\exp\left\{-\log{(\sigma_2)}\frac{(\eta-1)^2\left(\sigma_1\sigma_2\eta t-\delta\sigma_1\sigma_2\eta t-\delta\right)^2}{4C^2\sigma_1^2\sigma_2^2\eta^2}m\right\}.
    \end{align*}
    Pick $\delta=\frac{\sigma_1\sigma_2\eta t}{4}$ so that $\sigma_1\sigma_2\eta t-\delta \sigma_1\sigma_2\eta t-\delta\geq\frac{\sigma_1\sigma_2\eta t}{4}$, then
        $$\mathbb{P}_{\overline{X}}\left(s_n\left(\theta\right)\leq\left(1-\sigma_1\sigma_2\eta t\right)\sqrt{\frac{m}{n}}\right)\leq2\exp\left\{n\log{\left(\frac{12}{\sigma_1\sigma_2\eta t}\right)}-\log{(\sigma_2)}\frac{(\eta-1)^2t^2}{16C^2}m\right\}.$$
    Fix $\sigma_1=\sigma_2=\eta=\left(\frac{e}{e-1}\right)^{\frac{1}{12}}$ and fix $t=\left(\frac{e-1}{e}\right)^{\frac{1}{2}}$. Recall that $m> C_2n$, then we take
        $$C_2=\frac{384C^2\log{12\left(\frac{e}{e-1}\right)^{\frac{1}{4}}}}{\left[\left(\frac{e}{e-1}\right)^{\frac{1}{12}}-1\right]^2\frac{e-1}{e}\log{\frac{e}{e-1}}},$$
    then
        $$\mathbb{P}_{\overline{X}}\left(s_n\left(\theta\right)\leq\left[1-\left(\frac{e-1}{e}\right)^{\frac{1}{4}}\right]\sqrt{\frac{m}{n}}\right)\leq2\exp\left\{-\log{12\left(\frac{e}{e-1}\right)^{\frac{1}{4}}}n\right\}.$$
    By \textsc{Lemma} \ref{Grassmannian}, the event $\mathcal{F}$ in $\mathbb{O}\left(n^{\ell}\right)$ determines a subset $S_{\mathcal{F}}\subset\mathbf{G}_{n^{\ell},m}$ with measure at least $1-e^{-cn}$, for every $F\in S_{\mathcal{F}}$,
    \begin{align*}
        \mathbb{P}_{\overline{X}}\left(\overline{\mathcal{T}}_{\varepsilon}\right)=&\mathbb{P}_{\overline{X}}\left(\norm{\Pi_F\otimes_{j=1}^{\ell}\overline{X}^{(j)}}_2\leq \varepsilon\sqrt{m}\right)\\
        \leq&\mathbb{P}_{\overline{X}}\left(\norm{\Pi_F\otimes_{j=1}^{\ell}\overline{X}^{(j)}}_2\leq\frac{\varepsilon}{1-\left(\frac{e-1}{e}\right)^{\frac{1}{4}}} s_n\sqrt{n}\right)+\mathbb{P}_{\overline{X}}\left(s_n\leq \left[1-\left(\frac{e-1}{e}\right)^{\frac{1}{4}}\right]\sqrt{\frac{m}{n}}\right)\\
        \leq&\left(C\varepsilon\mathcal{L}_n\right)^n+2e^{-cn},
    \end{align*}
    and by \textsc{Proposition} \ref{2-sided},
    \begin{align*}
        \mathbb{P}_X\left(\mathcal{T}_{\varepsilon}\right)\leq\mathbb{P}_{\overline{X}}\left(\overline{\mathcal{T}}_{\varepsilon}\right)\mathbb{P}_X\left(E_t\right)+\mathbb{P}_X\left(E_t^C\right)\leq\left(C\varepsilon\mathcal{L}_n\right)^n+e^{-\frac{c\sqrt{n}}{C_P}}.
    \end{align*}

    \textbf{Case 3}: Suppose $C_1n\leq m\leq C_2n$. We partition the $\theta$ into $N=\frac{m}{C_1n}\leq\frac{C_2}{C_1}$ matrices $\theta_1,\cdots,\theta_N$ of dimension $C_1n\times n$ as follows
        $$\theta=\left(\theta_1^T,\cdots,\theta_N^T\right)^T.$$
    Then
        $$\norm{\Pi_F\otimes_{j=1}^{\ell}\overline{X}^{(j)}}_2^2=\norm{\theta X^{(1)}}_2^2=\sum_{p=1}^N\norm{\theta_pX^{(1)}}_2^2.$$
    By a union bound, there exists a subset $S_{\mathcal{F}}\subset\mathbf{G}_{n^{\ell},m}$, which is an intersection of subsets where $\mathbb{P}\left(\norm{\theta_p\overline{X}^{(1)}}_2\leq \varepsilon\sqrt{C_1n}\right)\leq\left(C\varepsilon\mathcal{L}_{C_1n}\right)^{C_1n}+2e^{-cC_1n}$ for $1\leq p\leq N$, with measure at least $1-\frac{C_2}{C_1}e^{-cC_1n}$, such that for every $F\in S_{\mathcal{F}}$,
    \begin{align*}
        \mathbb{P}_{\overline{X}}\left(\overline{\mathcal{T}}_{\varepsilon}\right)=&\mathbb{P}\left(\norm{\Pi_F\otimes_{j=1}^{\ell}\overline{X^{(j)}}_2\leq \varepsilon\sqrt{m}}\right)\\
        \leq&\mathbb{P}\left(\sum_{p=1}^N\norm{\theta_p\overline{X}^{(1)}}_2^2\leq \varepsilon^2m\right)\\
        \leq&\sum_{p=1}^N\mathbb{P}\left(\norm{\theta_p\overline{X}^{(1)}}_2\leq \varepsilon\sqrt{C_1n}\right)\\
        \leq&\frac{C_2}{C_1}\left(\left(C\varepsilon\mathcal{L}_{C_1n}\right)^{C_1n}+2e^{-cC_1n}\right),
    \end{align*}
    and by \textsc{Proposition} \ref{2-sided},
    \begin{align*}
        \mathbb{P}_X\left(\mathcal{T}_{\varepsilon}\right)\leq\mathbb{P}_{\overline{X}}\left(\overline{\mathcal{T}}_{\varepsilon}\right)\mathbb{P}_X\left(E_t\right)+\mathbb{P}_X\left(E_t^C\right)\leq\frac{C_2}{C_1}\left(C\varepsilon\mathcal{L}_{C_1n}\right)^{C_1n}+e^{-\frac{c\sqrt{n}}{C_P}}.
    \end{align*}
    That concludes the proof of \textsc{Theorem} \ref{1.5}.
\end{proof}

\section{Proof of \textsc{Theorem} \ref{1.1} and \ref{1.2}}
Now we first prove \textsc{Theorem} \ref{1.1}. 
\begin{proof}
Let $X^{(j)}\in\mathbb{R}^{n_j}, 1\leq j\leq\ell$ be independent random vectors with independent coordinates whose densities have uniform norms bounded by $M>0$ and let $z_1,\cdots,z_{\ell}\in\mathbb{R}^n$. Suppose $F$ is a subspace in $\mathbb{R}^{n_1\otimes\cdots\otimes n_{\ell}}$ with dimension $m$. Note that $2\sqrt{3}\cdot\mathbf{1}_{\left[-\frac{1}{2},\frac{1}{2}\right]^{n_j}}$ is isotropic log-concave. And $K_r=\left\{X\in\mathbb{R}^{n_1\otimes\cdots\otimes n_{\ell}}:\norm{\textbf{P}_FX}_2\leq r\right\}$ is a convex set in $\mathbb{R}^{n_1\otimes\cdots\otimes n_{\ell}}$. We apply \textsc{Theorem} \ref{1.4} for independent uniform distributions $2\sqrt{3}\cdot\mathbf{1}_{\left[-\frac{1}{2},\frac{1}{2}\right]^{n_j}}$, and we have
\begin{align*}
    \mathbb{P}\left(\norm{\Pi_F\otimes_{j=1}^{\ell}\left(2\sqrt{3}\cdot\mathbf{1}_{\left[-\frac{1}{2},\frac{1}{2}\right]^{n_j}}\right)}_2\leq\varepsilon\sqrt{m}\right)\leq\min\left\{m,{C'}^{\ell}\log\frac{1}{\varepsilon}\right\}\frac{\varepsilon}{(\ell-1)!}\left(C\log\frac{1}{\varepsilon}\right)^{\ell-1}.
\end{align*}
Notice that by \textsc{Corollary} \ref{3.4}, for every $M>0$,
$$\otimes_{j=1}^{\ell}M\left(X^{(j)}-z_j\right)\prec\otimes_{j=1}^{\ell}\mathbf{1}_{{\left[-\frac{1}{2},\frac{1}{2}\right]}^{n_j}}.$$ 
Hence we have for $0<\varepsilon<e^{-c{\ell}}$,
\begin{align*}
    &\mathbb{P}\left(\norm{\Pi_F\otimes_{j=1}^{\ell}\left(X^{(j)}-z_j\right)}_2\leq\frac{1}{\left(2\sqrt{3}M\right)^{\ell}}\varepsilon\sqrt{m}\right)\\
    \leq&\mathbb{P}\left(\norm{\Pi_F\otimes_{j=1}^{\ell}\cdot\mathbf{1}_{\left[-\frac{1}{2},\frac{1}{2}\right]^{n_j}}}_2\leq\frac{1}{\left(2\sqrt{3}\right)^{\ell}}\varepsilon\sqrt{m}\right)\\
    =&\mathbb{P}\left(\norm{\Pi_F\otimes_{j=1}^{\ell}\left(2\sqrt{3}\cdot\mathbf{1}_{\left[-\frac{1}{2},\frac{1}{2}\right]^{n_j}}\right)}_2\leq\varepsilon\sqrt{m}\right)\\
    \leq&\min\left\{m,{C'}^{\ell}\log\frac{1}{\varepsilon}\right\}\frac{\varepsilon}{(\ell-1)!}\left(C\log\frac{1}{\varepsilon}\right)^{\ell-1}.
\end{align*}

\end{proof}
Then we prove \textsc{Theorem} \ref{1.2}.
\begin{proof}
Notice that $2\sqrt{3}\cdot\mathbf{1}_{\left[-\frac{1}{2},\frac{1}{2}\right]^{n}}$ is isotropic log-concave, whose isotropic constant and Poincar\'e constant are bounded from above by universal constants. And notice that 
    $$\norm{\otimes_{j=1}^{\ell}2\sqrt{3}\cdot\mathbf{1}_{\left[-\frac{1}{2},\frac{1}{2}\right]^{n}}}_2=\prod_{j=1}^{\ell}\norm{2\sqrt{3}\cdot\mathbf{1}_{\left[-\frac{1}{2},\frac{1}{2}\right]^{n}}}_2.$$
Also notice that $2\sqrt{3}\cdot\mathbf{1}_{\left[-\frac{1}{2},\frac{1}{2}\right]^{n}}$ is sub-gaussian by Hoeffding inequality (see \textsc{Theorem} 2.2.2 in \cite{Vershynin_2018}). We apply \textsc{Theorem} \ref{1.5} and \textsc{Remark} \ref{subgaussian} for independent uniform distributions $2\sqrt{3}\cdot\mathbf{1}_{\left[-\frac{1}{2},\frac{1}{2}\right]^{n}}$. Then there exists a subset $\mathcal{S}_{\mathcal{F}}$ in $G_{n^{\ell},m}$ with Haar measure at least $1-e^{-c\max\left\{m,n\right\}}$, for every subspace $F\in\mathcal{S}_{\mathcal{F}}$, we have for $0<\varepsilon<1$,
\begin{align*}
    &\mathbb{P}\left(\norm{\Pi_F\otimes_{j=1}^{\ell}\left(2\sqrt{3}\cdot\mathbf{1}_{\left[-\frac{1}{2},\frac{1}{2}\right]^{n}}\right)}_2\leq\varepsilon\sqrt{m}\right)\leq(C\varepsilon)^{C'\min\left\{m,n\right\}}+e^{-C''n}.
\end{align*}
Let $X^{(1)},\cdots,X^{(\ell)}\in\mathbb{R}^n$ be independent random vectors with independent coordinates whose densities have uniform norms bounded by $M$. Let $z_1,\cdots,z_{\ell}\in\mathbb{R}^{n}$ be arbitrary vectors and let $m\leq n^{\ell}$. Notice again by \textsc{Corollary} \ref{3.4}, for every $M>0$,
$$\otimes_{j=1}^{\ell}M\left(X^{(j)}-z_j\right)\prec\otimes_{j=1}^{\ell}\mathbf{1}_{{\left[-\frac{1}{2},\frac{1}{2}\right]}^{n_j}}.$$ 
Hence we have
\begin{align*}
    &\mathbb{P}\left(\norm{\Pi_F\otimes_{j=1}^{\ell}\left(X^{(j)}-z_j\right)}_2\leq\frac{1}{\left(2\sqrt{3}M\right)^{\ell}}\varepsilon\sqrt{m}\right)\\
    \leq&\mathbb{P}\left(\norm{\Pi_F\otimes_{j=1}^{\ell}\left(2\sqrt{3}\cdot\mathbf{1}_{\left[-\frac{1}{2},\frac{1}{2}\right]^{n}}\right)}_2\leq\varepsilon\sqrt{m}\right)\\
    \leq&(C\varepsilon)^{C'\min\left\{m,n\right\}}+e^{-C''n}.
\end{align*}
\end{proof}

\section{Application}
Suppose we want to retrieve the component vectors with Gaussian perturbation. We first introduce the following lemma. This has been essentially proved in \cite{invertibility}. For completeness we provide a proof.

\begin{lemma}\label{smin}
    Let $X$ be an $r\times d$ matrix with $rank(X)=r$. Let $v_1^T,\cdots,v_r^T\in\mathbb{R}^d$ be the row vectors of $X$ and let 
        $$V_i=span\left\{v_j: 1\leq j\leq r, j\neq i\right\},\quad 1\leq i\leq r.$$
    Then
        $$\norm{X^{-1}}_{HS}^2=\sum_{\tau\subset[r],|\tau|=r-1}\frac{\det{\left(X_{\tau}\left(X_{\tau}^T\right)\right)}}{\det{\left(XX^T\right)}}=\sum_{i=1}^d\frac{1}{\norm{\Pi_{V_i^{\perp}}v_i}_2^2},$$ 
    where $X^{-1}$ is the  Moore-Penrose inverse. 
\end{lemma}
\begin{proof}
    Since $rank(X)=r$, $XX^T$ is invertible,
        $$X^{-1}=X^T\left(XX^T\right)^{-1}.$$
    Denote by $d_i$ the $i$-th entry on the diagonal of $XX^T$ and denote by $X_i$ the submatrix obtained by removing the $i$-th row of $X$ for $1\leq i\leq r$. Then
        $$d_i=\frac{\det{\left(X_iX_i^T\right)}}{\det{\left(XX^T\right)}}.$$
    Hence
        $$\norm{X^{-1}}_2^2=tr\left(\left(XX^T\right)^{-1}XX^T\left(XX^T\right)^{-1}\right)=tr\left(\left(XX^T\right)^{-1}\right)=\sum_{i=1}^r\frac{\det{\left(X_iX_i^T\right)}}{\det{\left(XX^T\right)}}.$$
    Denote by $[0,1]^r$ the $r$-dimensional unit cube and denote by 
        $$X[0,1]^r=\left\{\sum_{i=1}^ra_iv_i^T:0\leq a_1,\cdots,a_r\leq1\right\}$$ 
    the parallelepiped generated by $X$. For $1\leq i\leq r$ and $1\leq j\leq r$, define
        $$W_i^{(k)}=span\left\{v_j,1\leq j\leq k, j\neq i\right\}.$$
    Then for $1\leq i\leq r$ and $1\leq j\leq r$,
    \begin{align*}
        vol\left(X[0,1]^r\right)=&\det{\left(XX^T\right)}^{\frac{1}{2}}\\
        =&\norm{v_1}_2\norm{\Pi_{{W_{i}^{(1)}}^{\perp}}v_2}_2\cdots\norm{\Pi_{{W_{i}^{(i-2)}}^{\perp}}v_{i-1}}_2\norm{\Pi_{{W_{i}^{(i)}}^{\perp}}v_{i+1}}_2\cdots\norm{\Pi_{{W_{i}^{(r-1)}}^{\perp}}v_r}_2\norm{\Pi_{{W_{i}^{(r)}}^{\perp}}v_i}_2
    \end{align*}
    and
    \begin{align*}
        vol\left(X_i[0,1]^r\right)=&\det{\left(X_iX_i^T\right)}^{\frac{1}{2}}\\
        =&\norm{v_1}_2\norm{\Pi_{{W_{i}^{(1)}}}^{\perp}v_2}_2\cdots\norm{\Pi_{{W_{i}^{(i-2)}}^{\perp}}v_{i-1}}_2\norm{\Pi_{{W_{i}^{(i)}}^{\perp}}v_{i+1}}_2\cdots\norm{\Pi_{{W_{i}^{(r-1)}}^{\perp}}v_r}_2.
    \end{align*}
    Note that $W_{i}^{(r)}=V_{i}$, then
    $$\sum_{i=1}^r\frac{\det{\left(X_iX_i^T\right)}}{\det{\left(XX^T\right)}}=\sum_{i=1}^r\frac{1}{\norm{\Pi_{{W_{i}^{(r)}}^{\perp}}v_i}_2^2}=\sum_{i=1}^r\frac{1}{\norm{\Pi_{V_i^{\perp}}v_i}_2^2}.$$
\end{proof}

\begin{theorem}\label{application}
    Consider random vectors 
        $$\widetilde{X}_i^{(j)}=X_i^{(j)}+G_i^{(j)}\in\mathbb{R}^n, 1\leq i\leq r, 1\leq j\leq l$$
    where 
        $$\norm{X_i^{(j)}}_2\leq C,\quad G_i^{(j)}\sim N\left(0,\frac{\rho^2}{n}\mathbb{I}_n\right).$$
    Let $r\leq \frac{1}{2}n^{\ell}$. Define the $n^{\ell}\times r$ matrix $A$ where the $i-$th column of $A$ is the flattened vector of $\otimes_{j=1}^{\ell}\widetilde{X}_i^{(j)}$ for $1\leq i\leq r$. Then for $0<\epsilon<e^{-C\ell}$,
        $$\mathbb{P}\left(s_{\min}(A)\leq\sqrt{1-\frac{r}{n^{\ell}}}\left(c\rho\right)^{\ell}\varepsilon\right)\leq\frac{\varepsilon r}{(\ell-1)!}\left(C'\log\frac{1}{\varepsilon}\right)^{\ell}.$$
\end{theorem}

\begin{remark}
    The result is true if we replace the Gaussian random vectors $G_i^{(j)}$'s with any random vectors $Y_i^{(j)}$'s with independent coordinates whose densities are bounded by $\frac{C\sqrt{n}}{ \rho}$. Moreover, a similar statement is true if we replace $G_i^{(j)}$'s with centered log-concave random vectors $W_i^{(j)}$'s with covariance matrix $\frac{ \rho^{2}}{ n}\mathbb{I}_{n}$. We omit the details.
\end{remark}

\begin{proof}
    Let $v_i=\otimes_{j=1}^{\ell}\widetilde{X}_i^{(j)}$ for $1\leq i\leq r$ be simple random tensors. If we view $v_i$'s as flattened vectors, then $A^T$ is $r\times n^{\ell}$ matrix and $v_i$ has independent coordinates with densities bounded by $\frac{\sqrt{n}}{\sqrt{2\pi}\rho}$. Let 
        $$V_i=span\left\{v_j: 1\leq j\leq r, j\neq i\right\},\quad 1\leq i\leq r.$$
    Then $v_i$ is independent of $V_i$. Hence
        $$\norm{\left(A^T\right)^{-1}}_{HS}^2=\sum_{i=1}^r\frac{1}{s_i\left(A^T\right)^2}=\sum_{i=1}^r\frac{1}{\norm{\Pi_{V_i^{\perp}}v_i}_2^2},$$
    where $dim\left(V_i^{\perp}\right)=n^{\ell}-\lambda n^{\ell}+1$.
    Let $0\leq q\leq 1,$ then
    \begin{align*}
        \mathbb{E}\frac{1}{s_{\min}^q(A)}\leq&\mathbb{E}\left[\sum_{i=1}^r\frac{1}{\norm{\Pi_{V_i^{\perp}}v_i}_2^2}\right]^{\frac{q}{2}}=\mathbb{E}\left[\sum_{i=1}^r\left(\frac{1}{\norm{\Pi_{V_i^{\perp}}v_i}_2^q}\right)^{\frac{2}{q}}\right]^{\frac{q}{2}}\\
        \leq&\mathbb{E}\sum_{i=1}^r\frac{1}{\norm{\Pi_{V_i^{\perp}}v_i}_2^q}=\sum_{i=1}^r\mathbb{E}\frac{1}{\norm{\Pi_{V_i^{\perp}}v_i}_2^q}\\
        =&\sum_{i=1}^r\mathbb{E}_{v_j,j\neq i}\mathbb{E}_{v_i}\left[\left.\frac{1}{\norm{\Pi_{V_i^{\perp}}v_i}_2^q}\right| v_j,j\neq i\right],
    \end{align*}
    where the second line follows from the fact that $\mathbb{B}_p^r\supset\mathbb{B}_1^r$ for $p>1$. For $0<\epsilon<e^{-C\ell}$, by \textsc{Theorem} \ref{1.4},
    \begin{align*}
        &\mathbb{P}\left(\norm{\Pi_{V_i^{\perp}}v_i}_2\leq\left(\frac{\sqrt{2\pi}\rho}{2\sqrt{3}\sqrt{n}}\right)^{\ell}\varepsilon\sqrt{n^{\ell}-\lambda n^{\ell}+1}\right)\\
        =&\mathbb{P}\left(\norm{\Pi_{V_i^{\perp}}v_i}_2\leq\sqrt{1-\frac{r}{n^{\ell}}}\left(c\rho\right)^{\ell}\varepsilon\right)\\
        \leq&\frac{\varepsilon}{(\ell-1)!}\left(C'\log\frac{1}{\varepsilon}\right)^{\ell}.
    \end{align*}
    or for $0<\varepsilon<\frac{e^{-C\ell}}{\left(c\rho\right)^{\ell}\sqrt{1-\frac{r}{n^{\ell}}}}$,
    $$\mathbb{P}\left(\norm{\Pi_{V_i^{\perp}}v_i}_2\leq\varepsilon\right)\leq\frac{\varepsilon}{(\ell-1)!\left(c\rho\right)^{\ell}\sqrt{1-\frac{r}{n^{\ell}}}}\left(C''\log\frac{1}{\varepsilon}\right)^{\ell}.$$
    By \textsc{Lemma} \ref{C1},
    \begin{align*}
        \mathbb{E}_{v_j,j\neq i}\mathbb{E}_{v_i}\left[\left.\frac{1}{\norm{\Pi_{V_i^{\perp}}v_i}_2^q}\right| v_j,j\neq i\right]\leq\left(\frac{\ell^{\ell}}{(\ell-1)!\left(c\rho\right)^{\ell}\sqrt{1-\frac{r}{n^{\ell}}}}\right)^q\frac{1}{(1-q)^{q\left(\ell-1\right)+1}}.
    \end{align*}
    And by \textsc{Corollary} \ref{C2}, for $0<\varepsilon<\frac{e^{-C\ell}}{\left(c\rho\right)^{\ell}\sqrt{1-\frac{r}{n^{\ell}}}}$,
        $$\mathbb{P}\left(s_{\min}(A)\leq\epsilon\right)\leq\frac{\varepsilon r}{(\ell-1)!\left(c\rho\right)^{\ell}\sqrt{1-\frac{r}{n^{\ell}}}}\left(C''\log\frac{1}{\varepsilon}\right)^{\ell},$$
    or for $0<\epsilon<e^{-C\ell}$,
        $$\mathbb{P}\left(s_{\min}(A)\leq\sqrt{1-\frac{r}{n^{\ell}}}\left(c\rho\right)^{\ell}\varepsilon\right)\leq\frac{\varepsilon r}{(\ell-1)!}\left(C'\log\frac{1}{\varepsilon}\right)^{\ell}.$$        
\end{proof}

Recall the definition of Kahtri-Rao product. Let $U=\left(u_1,\cdots,u_r\right)$ be an $m\times r$ matrix and let $V=\left(v_1,\cdots,v_r\right)$ be an $n\times r$ matrix, then the Kahtri-Rao product of $U$ and $V$ is an $mn\times r$ matrix $U\odot V$ whose $i$-th column is the flattened vector of the tensor $u_i\otimes v_i$.

There exists an efficient algorithm (see \cite{smoothed}, \cite{three}) to decompose an order-$3$ tensor
    $$X=\sum_{i=1}^rU_i\otimes V_i\otimes W_i.$$
using simultaneous diagonalization when the rank $r$ is no more than the dimension $n$. For tensor with order more than $3$, we can write
    $$\widetilde{X}=\sum_{i=1}^r\otimes_{j=1}^{\ell}\widetilde{X}_i^{(j)}=\sum_{i=1}^ru_i\otimes v_i\otimes w_k,$$
where
    $$u_i=\widetilde{X}^{(1)}\otimes\cdots\otimes \widetilde{X}^{\left(\left\lfloor\frac{\ell-1}{2}\right\rfloor\right)},$$
    $$v_i=\widetilde{X}^{\left(\left\lfloor\frac{\ell-1}{2}\right\rfloor+1\right)}\otimes\cdots\otimes \widetilde{X}^{(\ell-1)},$$
    $$w_i=\widetilde{X}^{(\ell)}_k.$$
Then the algorithm succeeds with high probability if the smallest singular values of $U=\left(u_1,\cdots,u_r\right)$, $V=\left(v_1,\cdots,v_r\right)$ and $W=\left(w_1,\cdots,w_r\right)$ are sufficiently large with high probability. And the running time of the algorithm also depends on the smallest singular values. We refer to Section 2 in \cite{smoothed} and Appendix A in \cite{tensordecomp} for details. Our result in \textsc{Theorem} \ref{application} improves the previous results since we conclude that the median of the singular value is of order $\mathcal{O}_{\ell}\left(\frac{ {\rho}^{\ell}(\log{er})^{\ell} }{ r}\right)$ where  $r$ is the rank. In comparison, the best known result due to \cite{tensordecomp}  is of order $\mathcal{O}_{\ell}\left(\frac{\rho^{\ell}}{n^{\ell}\sqrt{r}}\right)$. We want to emphasize that our result depends only on the rank and not on the dimension of the component vectors and allows $\rho$ (which measures how large the noise should be to guarantee that the algorithm works) to be significantly less. Our results combined with the known algorithms imply the following. 

\begin{corollary}
For $r\leq \frac{n^{\frac{\ell-1}{2}}}{2}$, suppose we are given $\widetilde{T}+E$ where
$\widetilde{T}$ and $\mathcal{F}$ are order $\ell$-tensors and $\widetilde{T}$ has rank $r$ and is obtained from the above smoothed analysis model. Moreover, for $0<\delta<e^{-C\ell}$ and $\varepsilon<1$, suppose the entries of E are at most $\varepsilon\cdot\text{poly}\left(\frac{\left(c\rho\right)^{\ell}\delta}{n^{\frac{\ell-1}{2}}},\frac{\rho}{n},\frac{1}{n^{\ell}}\right)$. Then there is an algorithm to recover the rank one terms of $\widetilde{T}$ up to an additive $\varepsilon$ error. The algorithm runs in time $\text{poly}\left(\frac{n^{\frac{\ell-1}{2}}}{\left(c\rho\right)^{\ell}\delta},\frac{n}{\rho},n^{\ell}\right)$ and succeeds with probability at least $1-C_{\ell}\delta r\left(\log{\frac{e}{\delta}}\right)^{\ell+1}$.
\end{corollary}

\bibliographystyle{abbrv}
\bibliography{ref}

\begin{thebibliography}{10}

\bibitem{AdamczakLatalaMelle}
R.~Adamczak, R.~Latala, and R.~Meller.
\newblock Moments of gaussian chaoses in banach spaces.
\newblock {\em Electronic Journal of Probability}, 26, 01 2021.

\bibitem{AdamczakandLatala}
R.~Adamczak and R.~Latała.
\newblock {Tail and moment estimates for chaoses generated by symmetric random
  variables with logarithmically concave tails}.
\newblock {\em Annales de l'Institut Henri Poincaré, Probabilités et
  Statistiques}, 48(4):1103 -- 1136, 2012.

\bibitem{Adamczak2013ConcentrationIF}
R.~Adamczak and P.~Wolff.
\newblock Concentration inequalities for non-lipschitz functions with bounded
  derivatives of higher order.
\newblock {\em Probability Theory and Related Fields}, 162:531--586, 2013.

\bibitem{anandkumar2012method}
A.~Anandkumar, D.~Hsu, and S.~M. Kakade.
\newblock A method of moments for mixture models and hidden markov models.
\newblock In {\em Conference on learning theory}, pages 33--1. JMLR Workshop
  and Conference Proceedings, 2012.

\bibitem{pmlr-v23-anandkumar12}
A.~Anandkumar, D.~Hsu, and S.~M. Kakade.
\newblock A method of moments for mixture models and hidden markov models.
\newblock In S.~Mannor, N.~Srebro, and R.~C. Williamson, editors, {\em
  Proceedings of the 25th Annual Conference on Learning Theory}, volume~23 of
  {\em Proceedings of Machine Learning Research}, pages 33.1--33.34, Edinburgh,
  Scotland, 25--27 Jun 2012. PMLR.

\bibitem{tensordecomp}
N.~Anari, C.~Daskalakis, W.~Maass, C.~H. Papadimitriou, A.~Saberi, and
  S.~Vempala.
\newblock Smoothed analysis of discrete tensor decomposition and assemblies of
  neurons.
\newblock In {\em Proceedings of the 32nd International Conference on Neural
  Information Processing Systems}, NIPS'18, page 10880–10890, Red Hook, NY,
  USA, 2018. Curran Associates Inc.

\bibitem{AGGbook1}
S.~Artstein-Avidan, A.~Giannopoulos, and V.~Milman.
\newblock {\em Asymptotic geometric analysis}.
\newblock Number pt. 1 in Mathematical surveys and monographs. American
  Mathematical Society, 2015.

\bibitem{AGGbook2}
S.~Artstein-Avidan, A.~Giannopoulos, and V.~Milman.
\newblock {\em Asymptotic Geometric Analysis, Part II}.
\newblock Mathematical Surveys and Monographs. American Mathematical Society,
  2021.

\bibitem{bamberger2021hansonwright}
S.~Bamberger, F.~Krahmer, and R.~Ward.
\newblock The hanson–wright inequality for random tensors.
\newblock {\em Sampling Theory Signal Processing and Data Analysis}, 20, 2022.

\bibitem{KroneckerFJLToptimal}
S.~Bamberger, F.~Krahmer, and R.~Ward.
\newblock Johnson–lindenstrauss embeddings with kronecker structure.
\newblock {\em SIAM Journal on Matrix Analysis and Applications},
  43(4):1806--1850, 2022.

\bibitem{smoothed}
A.~Bhaskara, M.~Charikar, A.~Moitra, and A.~Vijayaraghavan.
\newblock Smoothed analysis of tensor decompositions.
\newblock In {\em Proceedings of the Forty-Sixth Annual ACM Symposium on Theory
  of Computing}, STOC '14, page 594–603, New York, NY, USA, 2014. Association
  for Computing Machinery.

\bibitem{bourgain1986geometry}
J.~Bourgain.
\newblock Geometry of banach spaces and harmonic analysis.
\newblock In {\em Proceedings of the International Congress of Mathematicians},
  volume~1, page~2. Citeseer, 1986.

\bibitem{bourgain1986high}
J.~Bourgain.
\newblock On high dimensional maximal functions associated to convex bodies.
\newblock {\em American Journal of Mathematics}, 108(6):1467--1476, 1986.

\bibitem{BLL}
H.~Brascamp, E.~Lieb, and J.~Luttinger.
\newblock A general rearrangement inequality for multiple integrals.
\newblock {\em Journal of Functional Analysis}, 17(2):227--237, Oct. 1974.
\newblock Funding Information: partially supported by National Science
  Foundation Grant GP-16147 A\#1 . partially supported by National Science
  Foundation Grant GP-31674 X. partially supported by a grant from the National
  Science Foundation.

\bibitem{Isotropic}
S.~Brazitikos, A.~Giannopoulos, P.~Valettas, and B.-H. Vritsiou.
\newblock {\em Geometry of isotropic convex bodies}, volume 196.
\newblock American Mathematical Soc., 2014.

\bibitem{polytope}
E.~M. Bronstein.
\newblock Approximation of convex sets by polytopes.
\newblock {\em Journal of Mathematical Sciences}, 153, 09 2008.

\bibitem{Burchard2009ASC}
A.~Burchard.
\newblock A short course on rearrangement inequalities.
\newblock {\em Lecture notes, IMDEA Winter School, Madrid}, 2009.

\bibitem{carberywright}
A.~Carbery and J.~Wright.
\newblock Distributional and l q norm inequalities for polynomials over convex
  bodies in $\mathbb{R}^n$.
\newblock {\em Mathematical Research Letters}, 8, 05 2001.

\bibitem{chang1996full}
J.~T. Chang.
\newblock Full reconstruction of markov models on evolutionary trees:
  identifiability and consistency.
\newblock {\em Mathematical biosciences}, 137(1):51--73, 1996.

\bibitem{Chang1996FullRO}
J.~T. Chang.
\newblock Full reconstruction of markov models on evolutionary trees:
  identifiability and consistency.
\newblock {\em Mathematical biosciences}, 137 1:51--73, 1996.

\bibitem{comon1994independent}
P.~Comon.
\newblock Independent component analysis, a new concept?
\newblock {\em Signal processing}, 36(3):287--314, 1994.

\bibitem{Comon1994IndependentCA}
P.~Comon.
\newblock Independent component analysis, a new concept?
\newblock {\em Signal Process.}, 36:287--314, 1994.

\bibitem{BoundingMarginal}
S.~Dann, G.~Paouris, and P.~Pivovarov.
\newblock Bounding marginal densities via affine isoperimetry.
\newblock {\em Proceedings of the London Mathematical Society},
  113(2):140--162, 2016.

\bibitem{dharmadhikari1988unimodality}
S.~Dharmadhikari and K.~Joag-Dev.
\newblock {\em Unimodality, convexity, and applications}.
\newblock Elsevier, 1988.

\bibitem{Fradelizi}
M.~Fradelizi.
\newblock Concentration inequalities for s-concave measures of dilations of
  borel sets and applications.
\newblock {\em Electronic Journal of Probability}, 14, 08 2008.

\bibitem{GLAZER2022109639}
I.~Glazer and D.~Mikulincer.
\newblock Anti-concentration of polynomials: Dimension-free covariance bounds
  and decay of fourier coefficients.
\newblock {\em Journal of Functional Analysis}, 283(9):109639, 2022.

\bibitem{invertibility}
E.~Gluskin and A.~Olevskii.
\newblock Invertibility of sub-matrices and the octahedron width theorem.
\newblock {\em Israel Journal of Mathematics}, 186, 11 2011.

\bibitem{Gotze2019ConcentrationIF}
F.~Gotze, H.~Sambale, and A.~Sinulis.
\newblock Concentration inequalities for polynomials in
  $\alpha$-sub-exponential random variables.
\newblock {\em Electronic Journal of Probability}, 2019.

\bibitem{guedon}
O.~Gu\'edon.
\newblock Kahane-khinchine type inequalities for negative exponent.
\newblock {\em Mathematika}, 46(1):165–173, 1999.

\bibitem{hsu2012spectral}
D.~Hsu, S.~M. Kakade, and T.~Zhang.
\newblock A spectral algorithm for learning hidden markov models.
\newblock {\em Journal of Computer and System Sciences}, 78(5):1460--1480,
  2012.

\bibitem{KroneckerFJLT}
R.~Jin, T.~G. Kolda, and R.~Ward.
\newblock {Faster Johnson–Lindenstrauss transforms via Kronecker products}.
\newblock {\em Information and Inference: A Journal of the IMA},
  10(4):1533--1562, 10 2020.

\bibitem{kannan1995isoperimetric}
R.~Kannan, L.~Lov{\'a}sz, and M.~Simonovits.
\newblock Isoperimetric problems for convex bodies and a localization lemma.
\newblock {\em Discrete \& Computational Geometry}, 13:541--559, 1995.

\bibitem{unimodality}
M.~Kanter.
\newblock Unimodality and dominance for symmetric random vectors.
\newblock {\em Transactions of The American Mathematical Society - TRANS AMER
  MATH SOC}, 229:65--65, 05 1977.

\bibitem{klartag2023logarithmic}
B.~Klartag.
\newblock Logarithmic bounds for isoperimetry and slices of convex sets.
\newblock {\em arXiv preprint arXiv:2303.14938}, 2023.

\bibitem{kls1}
B.~Klartag and J.~Lehec.
\newblock Bourgain’s slicing problem and kls isoperimetry up to polylog
  (2022).
\newblock {\em arXiv preprint arXiv:2203.15551}.

\bibitem{three}
J.~B. Kruskal.
\newblock Three-way arrays: rank and uniqueness of trilinear decompositions,
  with application to arithmetic complexity and statistics.
\newblock {\em Linear Algebra and its Applications}, 18(2):95--138, 1977.

\bibitem{Landsberg2011TensorsGA}
J.~M. Landsberg.
\newblock {\em Tensors: geometry and applications: geometry and applications},
  volume 128.
\newblock American Mathematical Soc., 2011.

\bibitem{latala}
R.~Latała.
\newblock {Estimation of moments of sums of independent real random variables}.
\newblock {\em The Annals of Probability}, 25(3):1502 -- 1513, 1997.

\bibitem{Ledoux2001TheCO}
M.~Ledoux.
\newblock The concentration of measure phenomenon.
\newblock {\em AMS Surveys and Monographs}, 89, 01 2001.

\bibitem{Lehec2011}
J.~Lehec.
\newblock {\em Moments of the Gaussian Chaos}, pages 327--340.
\newblock Springer Berlin Heidelberg, Berlin, Heidelberg, 2011.

\bibitem{orderthree}
S.~E. Leurgans, R.~T. Ross, and R.~B. Abel.
\newblock A decomposition for three-way arrays.
\newblock {\em SIAM Journal on Matrix Analysis and Applications},
  14(4):1064--1083, 1993.

\bibitem{LI2001533}
W.~Li and Q.-M. Shao.
\newblock Gaussian processes: Inequalities, small ball probabilities and
  applications.
\newblock In {\em Stochastic Processes: Theory and Methods}, volume~19 of {\em
  Handbook of Statistics}, pages 533--597. Elsevier, 2001.

\bibitem{lieb2001analysis}
E.~Lieb and M.~Loss.
\newblock {\em Analysis}.
\newblock Crm Proceedings \& Lecture Notes. American Mathematical Society,
  2001.

\bibitem{lim2021tensors}
L.-H. Lim.
\newblock Tensors in computations.
\newblock {\em Acta Numerica}, 30:555--764, 2021.

\bibitem{elizabeth}
E.~S. Meckes.
\newblock {\em The Random Matrix Theory of the Classical Compact Groups}.
\newblock Cambridge Tracts in Mathematics. Cambridge University Press, 2019.

\bibitem{mossel2005learning}
E.~Mossel and S.~Roch.
\newblock Learning nonsingular phylogenies and hidden markov models.
\newblock In {\em Proceedings of the thirty-seventh annual ACM symposium on
  Theory of computing}, pages 366--375, 2005.

\bibitem{MosselRoch}
E.~Mossel and S.~Roch.
\newblock {Learning nonsingular phylogenies and hidden Markov models}.
\newblock {\em The Annals of Applied Probability}, 16(2):583 -- 614, 2006.

\bibitem{nazarov2001local}
F.~Nazarov, M.~Sodin, and A.~Volberg.
\newblock Local dimension-free estimates for volumes of sublevel sets of
  analytic functions.
\newblock {\em arXiv preprint math/0108213}, 2001.

\bibitem{nazarov2002geometric}
F.~L. Nazarov, M.~L. Sodin, and A.~L. Volberg.
\newblock The geometric kannan--lov{\'a}sz--simonovits lemma, dimension-free
  estimates for volumes of sublevel sets of polynomials, and distribution of
  zeros of random analytic functions.
\newblock {\em Algebra i Analiz}, 14(2):214--234, 2002.

\bibitem{Nguyen2013}
H.~H. Nguyen and V.~H. Vu.
\newblock {\em Small Ball Probability, Inverse Theorems, and Applications},
  pages 409--463.
\newblock Springer Berlin Heidelberg, Berlin, Heidelberg, 2013.

\bibitem{Paouris2006ConcentrationOM}
G.~Paouris.
\newblock Concentration of mass on convex bodies.
\newblock {\em Geometric \& Functional Analysis GAFA}, 16:1021--1049, 2006.

\bibitem{smallball}
G.~Paouris.
\newblock Small ball probability estimates for log-concave measures.
\newblock {\em Transactions of the American Mathematical Society},
  364(1):287--308, 2012.

\bibitem{RandomConvexSets}
G.~Paouris and P.~Pivovarov.
\newblock Small-ball probabilities for the volume of random convex sets.
\newblock {\em Discrete \& Computational Geometry}, 04 2013.

\bibitem{randomized}
G.~Paouris and P.~Pivovarov.
\newblock Randomized isoperimetric inequalities.
\newblock In E.~Carlen, M.~Madiman, and E.~M. Werner, editors, {\em Convexity
  and Concentration}, pages 391--425, New York, NY, 2017. Springer New York.

\bibitem{Rogers}
C.~A. Rogers.
\newblock {A Single Integral Inequality}.
\newblock {\em Journal of the London Mathematical Society}, s1-32(1):102--108,
  01 1957.

\bibitem{linearimage}
M.~Rudelson and R.~Vershynin.
\newblock {Small Ball Probabilities for Linear Images of High-Dimensional
  Distributions}.
\newblock {\em International Mathematics Research Notices},
  2015(19):9594--9617, 12 2014.

\bibitem{smoothedmodel}
D.~A. Spielman and S.-H. Teng.
\newblock Smoothed analysis of algorithms: Why the simplex algorithm usually
  takes polynomial time.
\newblock {\em J. ACM}, 51(3):385–463, may 2004.

\bibitem{TaoVu}
T.~Tao and V.~Vu.
\newblock The littlewood-offord problem in high dimensions and a conjecture of
  frankl and füredi.
\newblock {\em Combinatorica}, 04 2012.

\bibitem{Vershynin_2018}
R.~Vershynin.
\newblock {\em High-Dimensional Probability: An Introduction with Applications
  in Data Science}.
\newblock Cambridge Series in Statistical and Probabilistic Mathematics.
  Cambridge University Press, 2018.

\bibitem{Vershynin2019ConcentrationIF}
R.~Vershynin.
\newblock Concentration inequalities for random tensors.
\newblock {\em Bernoulli}, 2019.

\end{thebibliography}
\newpage

\appendix
\section{}\label{A}
We construct an example to show the sharpness of the estimate in \textsc{Theorem} \ref{1.2}

\begin{lemma}\label{A1}
    Let $X_1,\cdots,X_{\ell}$ be independent uniform random variables on $\left[-1,1\right]$. Let $Z_{\ell}=X_1X_2\cdots X_{\ell}$, then its distribution satisfies
        $$F_{Z_{\ell}}(z)=\frac{1}{2}+\frac{z}{2}\sum_{j=0}^{\ell}\frac{\left(\log{\frac{1}{z}}\right)^j}{j!}.$$
\end{lemma}
\begin{proof}
First consider $Z_2=X_1X_2$. Suppose $0<z\leq1$, then
\begin{align*}
    \mathbb{P}\left(Z_2\leq z\right)=&\mathbb{E}_{X_2}\mathbb{P}_{X_1}\left(X_1X_2\leq z\mid X_2\right)\\
    =&\int_0^1\frac{1}{2}\dd y\int_{-1}^{\min\{\frac{z}{y},1\}}\frac{1}{2}\dd x+\int_{-1}^0\frac{1}{2}\dd y\int^1_{\max\{\frac{z}{y},-1\}}\frac{1}{2}\dd x\\
    =&\frac{1}{4}\left(\int_0^1\left(\min\{\frac{z}{y},1\}+1\right)\dd y+\int_{-1}^0\left(1-\max\{\frac{z}{y},-1\}\right)\dd y\right)\\
    =&\frac{1}{4}\left(2+\int_0^z\dd x+\int_z^1\frac{z}{x}\dd x+\int_{-1}^{-z}-\frac{z}{x}\dd x+\int_{-z}^0\dd x\right)\\
    =&\frac{1}{2}+\frac{z}{2}-\frac{z}{2}\log{z}.
\end{align*}
If $-1\leq z<0$, then by symmetry of uniform distribution,
\begin{align*}
    \mathbb{P}\left(Z_2\leq z\right)=&1-\mathbb{P}\left(Z_2> z\right)\\
    =&1-\mathbb{P}\left(Z_2<-z\right)\\
    =&1-\left(\frac{1}{2}+\frac{z}{2}+\frac{z}{2}\log{\left(-x\right)}\right)\\
    =&\frac{1}{2}+\frac{z}{2}-\frac{z}{2}\log{\left(-z\right)}.
\end{align*}
Hence for $-1\leq z\leq 1$, the distribution function of $Z_2$ is
$$F_{Z_2}(z)=\frac{1}{2}+\frac{z}{2}-\frac{z}{2}\log{|z|}.$$
Now we assume that for $k\geq2$ and $Z_k=X_1X_2\cdots X_k$, its distribution function is
$$F_{Z_k}(z)=\frac{1}{2}+\frac{z}{2}\sum_{j=0}^{k-1}\frac{\left(\log{\frac{1}{|z|}}\right)^j}{j!},$$
and denote its density fuction by $f_{Z_k}$.
Then for $Z_{k+1}=X_1X_2\cdots X_{k+1}$, suppose $0<z\leq1$,
\begin{align*}
    \mathbb{P}\left(Z_{k+1}\leq z\right)=&\mathbb{E}_{X_{k+1}}\mathbb{P}_{Z_k}\left(Z_kX_{k+1}\leq z\mid X_{k+1}\right)\\
    =&\int_0^1\frac{1}{2}\dd y\int_{-1}^{\min\{\frac{z}{y},1\}}f_{Z_k}(x)\dd x+\int_{-1}^0\frac{1}{2}\dd y\int^1_{\max\{\frac{z}{y},-1\}}f_{Z_k}(x)\dd x\\
    =&\frac{1}{2}\left(\int_0^1F_{Z_k}\left(\min\left\{\frac{z}{y},1\right\}\right)\dd y+\int_{-1}^0\left(1-F_{Z_k}\left(\max\left\{\frac{z}{y},-1\right\}\right)\right)\dd y\right)\\
    =&\frac{1}{2}\left(\int_0^z\dd y+\int_z^1F_{Z_k}\left(\frac{z}{y}\right)\dd y+\int_{-1}^{-z}\left(1-F_{Z_k}\left(\frac{z}{y}\right)\right)\dd y+\int_{-z}^0\dd y\right)\\
    =&\frac{1}{2}\left(2z+\int_z^1\left(\frac{1}{2}+\frac{z}{2y}\sum_{j=0}^{k-1}\frac{\left(\log{\left(\frac{y}{z}\right)}\right)^j}{j!}\right)\dd y+\int_{-1}^{-z}\left(\frac{1}{2}-\frac{z}{2y}\sum_{j=0}^{k-1}\frac{\left(\log{\left(-\frac{y}{z}\right)}\right)^j}{j!}\right)\dd y\right)\\
    =&\frac{1}{2}\left(2z+z\int_1^{\frac{1}{z}}\left(\frac{1}{2}+\frac{1}{2u}\sum_{j=0}^{k-1}\frac{\left(\log{u}\right)^j}{j!}\right)\dd u+z\int_{-\frac{1}{z}}^{-1}\left(\frac{1}{2}+\frac{1}{2u}\sum_{j=0}^{k-1}\frac{\left(\log{\left(-u\right)}\right)^j}{j!}\right)\dd u\right)\\
    =&z+\left.\frac{z}{2}\left(\frac{u}{2}+\frac{1}{2}\sum_{j=1}^{k}\frac{\left(\log{u}\right)^j}{j!}\right)\right|_{u=1}^{u=\frac{1}{z}}+\left.\frac{z}{2}\left(\frac{u}{2}-\frac{1}{2}\sum_{j=1}^{k}\frac{\left(\log{\left(-u\right)}\right)^j}{j!}\right)\right|_{u=-\frac{1}{z}}^{u=-1}\\
    =&\frac{1}{2}+\frac{z}{2}\sum_{j=0}^{k}\frac{\left(\log{\frac{1}{z}}\right)^j}{j!}
\end{align*}
\end{proof}
\begin{lemma}\label{A2}
    Let $X^{(1)},\cdots,X^{(\ell)}\in\mathbb{R}^n$ be independent uniform random vectors on $\left[-\sqrt{3},\sqrt{3}\right]^n$, where $X^{(k)}=\left(X_1^{(k)},\cdots,X_n^{(k)}\right)$. Then
        $$\mathbb{P}\left(\sum_{i=1}^m{X_i^{(1)}}^2{X_1^{(2)}}^2\cdots {X_1^{(\ell)}}^2\leq\varepsilon^2 m\right)\geq\frac{C\varepsilon}{(\ell-2)!}\left(\log{\frac{1}{\varepsilon}}\right)^{\ell-2}.$$
\end{lemma}
\begin{proof}
Note that
\begin{align*}
   &\mathbb{P}\left(\sum_{i=1}^m{X_i^{(1)}}^2{X_1^{(2)}}^2\cdots {X_1^{(\ell)}}^2\leq\varepsilon^2 m\right)\\
   \geq&\mathbb{E}_{X^{(1)}}\mathbf{1}_{\{\sum_{i=1}^m{X_i^{(1)}}^2\leq\varepsilon^2m\}}\mathbb{P}_{X^{(2)},\cdots,X^{(\ell)}}\left(\left|{X_1^{(2)}}\cdots {X_1^{(\ell)}}\right|\leq\frac{\varepsilon\sqrt{m}}{\left(\sum_{i=1}^m{X_i^{(1)}}^2\right)^{1/2}}\right)\\
   =&\mathbb{E}_{X^{(1)}}\frac{\varepsilon\sqrt{m}}{\sqrt{3}\left(\sum_{i=1}^m{X_i^{(1)}}^2\right)^{1/2}}\sum_{j=0}^{\ell-2}\frac{\left(\log\dfrac{\sqrt{3}\left(\sum_{i=1}^m{X_i^{(1)}}^2\right)^{1/2}}{\varepsilon\sqrt{m}}\right)^j}{j!}\\
   \geq&\mathbb{E}_{X^{(1)}}\frac{\varepsilon\sqrt{m}}{\sqrt{3}\left(\sum_{i=1}^m{X_i^{(1)}}^2\right)^{1/2}}\frac{\left(\log\dfrac{\sqrt{3}\left(\sum_{i=1}^m{X_i^{(1)}}^2\right)^{1/2}}{\varepsilon\sqrt{m}}\right)^{\ell-2}}{(\ell-2)!}.
\end{align*}
Denote $Y=\left(X^{(1)}_1,\cdots,X^{(1)}_m\right)$. Then $Y$ is subgaussian and $\norm{Y}_2=\left(\sum_{i=1}^m{X_i^{(1)}}^2\right)^{1/2}$. By concentration of norms for subgaussian random vectors,
    $$\mathbb{P}\left(\left|\norm{Y}_2-\mathbb{E}\norm{Y}_2\right|\leq t\sqrt{m}\right)\geq 1-e^{-Ct^2m}.$$
Note that $\mathbb{E}\norm{Y}_2^2=m$. Since $Y$ is log-concave, then by Borell's lemma
    $$\frac{1}{6}\sqrt{m}\leq\mathbb{E}\norm{Y}_2\leq\sqrt{m}$$
Then 
\begin{align*}
   &\mathbb{P}\left(\sum_{i=1}^m{X_i^{(1)}}^2{X_1^{(2)}}^2\cdots {X_1^{(\ell)}}^2\leq\varepsilon^2 m\right)\\
   \geq&\mathbb{E}_{X^{(1)}}\mathbf{1}_{\left\{\frac{1}{\sqrt{3}}\sqrt{m}\leq\norm{Y}_2\leq\left(\frac{7}{6}-\frac{1}{\sqrt{3}}\right)\sqrt{m}\right\}}\frac{\varepsilon\sqrt{m}}{\sqrt{3}\norm{Y}_2}\frac{\left(\log\dfrac{\sqrt{3}\norm{Y}_2}{\varepsilon\sqrt{m}}\right)^{\ell-2}}{(\ell-2)!}\\
   \geq&\left(1-e^{-\left(1-\frac{1}{\sqrt{3}}\right)^2Cm}\right)\frac{\varepsilon}{\sqrt{3}\left(\frac{7}{6}-\frac{1}{\sqrt{3}}\right)(\ell-2)!}\left(\log{\frac{1}{\varepsilon}}\right)^{\ell-2}\\
   =&\frac{C\varepsilon}{(\ell-2)!}\left(\log{\frac{1}{\varepsilon}}\right)^{\ell-2}.
\end{align*}  
\end{proof}
Let $X_1,\cdots,X_{\ell}$ be independent uniform distributions on $[-\sqrt{3},\sqrt{3}]$ such that $\mathbb{E}[X_j]=0$ and $\mathbf{Var}[X_j]=1$ for $1\leq j\leq\ell$. Define $Z=\prod_{j=1}^{\ell}X_j$. By \textsc{Lemma} \ref{A1}, the cumulative distribution function of $Z$ is
    $$F_Z(z)=\frac{1}{2}+\frac{z}{2\sqrt{3}}\sum_{j=0}^{\ell-1}\frac{\left(\log\frac{\sqrt{3}}{|z|}\right)^j}{j!}, \quad -\sqrt{3}\leq z\leq \sqrt{3}.$$
Choose unit vector $f\in\mathbb{R}^{n_1\otimes\cdots\otimes n_{\ell}}$ such that
    $$\displaystyle{f_{i_1\cdots i_{\ell}} = \left\{\begin{array}{cc} 1 & \text{ if } i_1=\cdots=i_{\ell} \\[0.5em] 0 & \text{ otherwise } \end{array} \right. }$$
so that $\left\langle X^{(1)}\otimes\cdots\otimes X^{(\ell)},f\right\rangle=X_1^{(1)}\cdots X_1^{(\ell)}$. For $0<\varepsilon<1$,
    $$\mathbb{P}[|\left\langle X^{(1)}\otimes\cdots\otimes X^{(\ell)},f\right\rangle|\leq \varepsilon]=F(\varepsilon)-F(-\varepsilon)=\frac{\varepsilon}{\sqrt{3}}\sum_{j=0}^{\ell-1}\frac{\left(\log\frac{\sqrt{3}}{\varepsilon}\right)^j}{j!}\geq \frac{1}{\sqrt{3}\left(\ell-1\right)!}\varepsilon\left(\log{\frac{\sqrt{3}}{\varepsilon}}\right)^{\ell-1}.$$
Choose unit vectors $f^{k}\in\mathbb{R}^{n_1\otimes\cdots\otimes n_{\ell}}$ for $1\leq k\leq m\leq n$ such that 
    $$\displaystyle{f_{i_1\cdots i_{\ell}} = \left\{\begin{array}{cc} 1 & \text{ if } i_1=k, i_2=\cdots=i_{\ell}=1 \\[0.5em] 0 & \text{ otherwise } \end{array} \right. }.$$
Then by \textsc{Lemma} \ref{A2},
\begin{align*}
    &\mathbb{P}\left(\norm{\Pi_FX^{(1)}\otimes\dots\otimes X^{(\ell)}}_2\leq\varepsilon\sqrt{m}\right)\\
   =&\mathbb{P}\left(\sum_{i=1}^m{X_i^{(1)}}^2{X_1^{(2)}}^2\cdots {X_1^{(\ell)}}^2\leq\varepsilon^2 m\right)\\
   \geq&\frac{C\varepsilon}{(\ell-2)!}\left(\log{\frac{1}{\varepsilon}}\right)^{\ell-2}.
\end{align*}

\section{}\label{B}
\begin{lemma}\label{B1}
    Let $\xi$ be a non-negative random variable, let $\ell\geq2$ be an integer and let $C\geq1$ be a universal constant. Assume for $0<\varepsilon<c_0\leq1$,
        $$\mathbb{P}\left(\xi\leq\varepsilon\right)\leq C\varepsilon.$$
    Let $c=\min\left\{c_0,e^{-1}\right\}$, for $0<\varepsilon<c^{\ell}$,
        $$\mathbb{E}\mathbf{1}_{\left\{\xi>c^{1-\ell}\varepsilon\right\}}\frac{\varepsilon}{\xi}\left(\log{\frac{\xi}{\varepsilon}}\right)^{\ell-2}\leq C\varepsilon\left(\log{\frac{1}{\varepsilon}}\right)^{\ell-1},$$
\end{lemma}
\begin{proof}
    Define $g(t)=t\left(\log{\frac{1}{t}}\right)^{\ell-2}$, then $g(t)\geq0$ for $0<t<1$. Also
        $$g'(t)=\left(\log{\frac{1}{t}}\right)^{\ell-3}\left(\log{\frac{1}{t}}+2-l\right)\geq0$$
    for $0<\varepsilon<e^{2-\ell}$. By using L'Hospital's Rule $\ell-2$ times we have $\displaystyle{\lim_{t\to0}g(t)=0}$. Notice that $c^{\ell-1}\leq e^{1-\ell}<e^{2-\ell}$. Hence we can write
        $$g(T)\mathbf{1}_{\left\{0<T<c^{\ell-1}\right\}}=\int_0^{T}g'(t)\mathbf{1}_{\left\{0<t<c^{\ell-1}\right\}}\dd t=\int_0^{\infty}g'(t)\mathbf{1}_{\left\{0<t<c^{\ell-1}\right\}}\mathbf{1}_{\left\{T\geq t\right\}}\dd t,$$
    and by Fubini-Tonelli's theorem,
        $$\mathbb{E}\left[g(T)\mathbf{1}_{\left\{0<T<c^{\ell-1}\right\}}\right]=\int_0^{\infty}g'(t)\mathbf{1}_{\left\{0<t<c^{\ell-1}\right\}}\mathbb{P}(T\geq t)\dd t=\int_0^{c^{\ell-1}}g'(t)\mathbb{P}(T\geq t)\dd t$$
    where $T$ is a non-negative random variable. Then for $0<\varepsilon<c^{\ell}$,
    \begin{align*}
        &\mathbb{E}\mathbf{1}_{\left\{\xi>c^{1-\ell}\varepsilon\right\}}\frac{\varepsilon}{\xi}\left(\log{\frac{\xi}{\varepsilon}}\right)^{\ell-2}\\
        =&\mathbb{E}\mathbf{1}_{\left\{0<\frac{\varepsilon}{\xi}<c^{\ell-1}\right\}}g\left(\frac{\varepsilon}{\xi}\right)\\
        =&\int_0^{c^{\ell-1}}g'(t)\mathbb{P}\left(\frac{\varepsilon}{\xi}\geq t\right)dt\\
        =&\int_0^{\frac{\varepsilon}{c}}\left(\log{\frac{1}{t}}\right)^{\ell-3}\left(\log{\frac{1}{t}}+2-l\right)\mathbb{P}\left(\frac{\varepsilon}{\xi}\geq t\right)dt+\int_{\frac{\varepsilon}{c}}^{c^{\ell-1}}\left(\log{\frac{1}{t}}\right)^{\ell-3}\left(\log{\frac{1}{t}}+2-l\right)\mathbb{P}\left(\frac{\varepsilon}{\xi}\geq t\right)dt\\
        \leq&\int_0^{\frac{\varepsilon}{c}}\left(\log{\frac{1}{t}}\right)^{\ell-3}\left(\log{\frac{1}{t}}+2-l\right)dt+\int_{\frac{\varepsilon}{c}}^{c^{\ell-1}}\left(\log{\frac{1}{t}}\right)^{\ell-3}\left(\log{\frac{1}{t}}+2-l\right)\frac{C\varepsilon}{t}dt\\
        =&\left.t\left(\log{\frac{1}{t}}\right)^{\ell-2}\right|_0^{\frac{\varepsilon}{c}}+\left.C\varepsilon\left[\left(\log{\frac{1}{t}}\right)^{\ell-2}-\frac{1}{\ell-1}\left(\log{\frac{1}{t}}\right)^{\ell-1}\right]\right|_{\frac{\varepsilon}{c}}^{c^{\ell-1}}\\
        \leq&\frac{\varepsilon}{c}\left(\log{\frac{c}{\varepsilon}}\right)^{\ell-2}+C\varepsilon\left[\frac{1}{\ell-1}\left(\log{\frac{c}{\varepsilon}}\right)^{\ell-1}-\left(\log{\frac{c}{\varepsilon}}\right)^{\ell-2}\right]\\
        \leq&\frac{C\varepsilon}{\ell-1}\left(\log{\frac{1}{\varepsilon}}\right)^{\ell-1}.
    \end{align*}
\end{proof}

\section{}\label{C}
\begin{lemma}\label{C1}
    Let $\xi$ be non-negative random variable, let $\ell\geq2$ and let $K\leq\left(\frac{M}{\ell}\right)^{\ell}$ be a positive parameter for some positive universal constant $M\geq1$. If there exists $c>\frac{1}{M}$, such that for $0\leq\varepsilon\leq c^{\ell}$, we have
        $$\mathbb{P}\left(\xi\leq\varepsilon\right)\leq K\varepsilon\left(\log{\frac{1}{\varepsilon}}\right)^{\ell-1},$$
    then for $1-\frac{1}{M_0}\leq q<1$, where $M_0\geq\frac{3}{2}$ is a constant that only depends on $M$, 
    \begin{equation*}
        \mathbb{E}\frac{1}{\xi^q}\leq\frac{\left(K\ell^{\ell}\right)^q}{(1-q)^{q\left(\ell-1\right)+1}}.
    \end{equation*}
\end{lemma}
\begin{proof}
    Note that for $0<a<c^{\ell}$,
    \begin{align*}
        \mathbb{E}\frac{1}{\xi^q}=&\int_0^{\infty}\mathbb{P}\left(\frac{1}{\xi^q}>u\right)\dd u\\
        =&q\int_0^{\infty}\frac{1}{t^{q+1}}\mathbb{P}\left(\xi<t\right)\dd t\\
        \leq&qK\int_0^a\frac{1}{t^q}\left(\log{\frac{1}{t}}\right)^{\ell-1}\dd t+q\int_a^{\infty}\frac{1}{t^{q+1}}\dd t\\
        =&qK\int_0^a\frac{1}{t^q}\left(\log{\frac{1}{t}}\right)^{\ell-1}\dd t+\frac{1}{a^q}
    \end{align*}
    Denote $\displaystyle I_p=\int_0^at^{-q}\left(\log{\frac{1}{t}}\right)^p\dd t$, then
        $$I_0=\int_0^at^{-q}\dd t=\frac{a^{1-q}}{1-q},$$
    and by integration by parts
    \begin{align*}
        I_p=&\frac{1}{1-q}\int_0^a\left(\log{\frac{1}{t}}\right)^p\dd t^{1-q}\\
        =&\frac{a^{1-q}}{1-q}\left(\log{\frac{1}{a}}\right)^p+\frac{p}{1-q}\int_0^at^{-q}\left(\log{\frac{1}{t}}\right)^{p-1}\dd t\\
        =&\frac{a^{1-q}}{1-q}\left(\log{\frac{1}{a}}\right)^p+\frac{p}{1-q}I_{p-1}.
    \end{align*}
    Hence
    \begin{align*}
        I_{\ell-1}=&\frac{a^{1-q}}{1-q}\left(\log{\frac{1}{a}}\right)^{\ell-1}+\frac{\ell-1}{1-q}I_{l-2}\\
        =&a^{1-q}\sum_{i=0}^{\ell-1}\frac{(\ell-1)!}{(\ell-1-i)!(1-q)^{i+1}}\left(\log{\frac{1}{a}}\right)^{\ell-1-i}.
    \end{align*}
    and
    \begin{align*}
        \mathbb{E}\frac{1}{\xi^q}\leq qKI_{\ell-1}+\frac{1}{a^q}=&a^{-q}\left(\frac{qKa}{1-q}\left(\log{\frac{1}{a}}\right)^{\ell-1}\sum_{i=0}^{\ell-1}\frac{(\ell-1)!}{(\ell-1-i)!}\left(\frac{1}{(1-q)\log{\frac{1}{a}}}\right)^i+1\right)\\
        \leq&a^{-q}\left(\frac{qKa}{1-q}\left(\log{\frac{1}{a}}\right)^{\ell-1}\sum_{i=0}^{\ell-1}\left(\frac{\ell}{(1-q)\log{\frac{1}{a}}}\right)^i+1\right)
    \end{align*}
    Note that $0<q<1$, and note that $\frac{\ell}{(1-q)\log{\frac{1}{a}}}=1$ when $a=e^{-\frac{\ell}{1-q}}$. Then
    \begin{equation*}
        \sum_{i=0}^{\ell-1}\left(\frac{\ell}{(1-q)\log{\frac{e}{a}}}\right)^i\leq
        \begin{cases}
            \ell^{\ell}\left(\frac{1}{\left(1-q\right)\log{\frac{1}{a}}}\right)^{\ell-1}&\text{if $a\geq e^{-\frac{\ell}{1-q}}$,}\\
            \ell&\text{if $a<e^{-\frac{\ell}{1-q}}$.}
        \end{cases}
    \end{equation*}
    Then
    \begin{equation*}
        \mathbb{E}\frac{1}{\xi^q}\leq
        \begin{cases}
            a^{-q}\left(\frac{qK\ell^{\ell}a}{(1-q)^{\ell}}+1\right)&\text{if $a\geq e^{-\frac{\ell}{1-q}}$,}\\
            a^{-q}\left[\frac{qK\ell a}{1-q}\left(\log{\frac{1}{a}}\right)^{\ell-1}+1\right]&\text{if $a<e^{-\frac{\ell}{1-q}}$,}\\
        \end{cases}
    \end{equation*}
    Take 
        $$f(a):=a^{-q}\left(\frac{qK\ell^{\ell}a}{(1-q)^{\ell}}+1\right)$$ 
    and take $a_0=\frac{(1-q)^{\ell-1}}{K\ell^{\ell}}$, then we claim that $a_0\geq e^{-\frac{\ell}{1-q}}$ if $q\geq1-\frac{1}{M_0}$ for some constant $M_0\geq\frac{3}{2}$. In fact, since $K\leq\left(\frac{M}{\ell}\right)^{\ell}$, there exists constant $M_0\geq\frac{3}{2}$ such that
        $$\left(K\ell^{\ell}\right)^{\frac{1}{\ell-1}}\leq\left(M^{\ell}\right)^{\frac{1}{\ell-1}}\leq M_0^{e-1}.$$
    If $q\geq1-\frac{1}{M_0}$, then
        $$\frac{1}{1-q}\geq M_0\geq\left(K\ell^{\ell}\right)^{\frac{1}{(e-1)(\ell-1)}},$$
        $$(e-1)(\ell-1)\log{\frac{1}{1-q}}\geq\log{\left(K\ell^{\ell}\right)}.$$
    Observe that
        $$\frac{1}{1-q}\geq e\log{\frac{1}{1-q}},$$
    therefore
    \begin{align*}
        \log{K\ell^{\ell}}\leq&(e-1)(\ell-1)\log{\frac{1}{1-q}}\\
        \leq&(\ell-1)\left(\frac{1}{1-q}-\log{\frac{1}{1-q}}\right)\\
        \leq&\frac{\ell}{1-q}-(\ell-1)\log{\frac{1}{1-q}},\\
        -\frac{\ell}{1-q}\leq&(\ell-1)\log{(1-q)}-\log{\left(K\ell^{\ell}\right)},\\
        e^{-\frac{\ell}{1-q}}\leq&\frac{(1-q)^{\ell-1}}{K\ell^{\ell}}=a_0.
    \end{align*}
    Hence 
        $$\mathbb{E}\frac{1}{\xi^q}\leq f\left(a_0\right)=\frac{\left(K\ell^{\ell}\right)^q}{(1-q)^{q\left(\ell-1\right)+1}}.$$   
\end{proof}
\begin{remark}\label{C2}
    \textsc{Lemma} \ref{C1} can be reversed for sufficiently small $\varepsilon$ at a price. Let $\xi$ be non-negative random variable, let $\ell\geq2$ and let $K\leq\left(\frac{M}{\ell}\right)^{\ell}$ be a positive parameter for some positive universal constant $M$ as assumed in \textsc{Lemma} \ref{C1}. If for $1-\frac{1}{M_0}\leq q<1$, where $M_0\geq\frac{3}{2}$ is a constant that only depends on $M$, we have
    \begin{equation*}
        \mathbb{E}\frac{1}{\xi^q}\leq\frac{\left(K\ell^{\ell}\right)^q}{(1-q)^{q\left(\ell-1\right)+1}},
    \end{equation*}
    then for $0<\varepsilon\leq\frac{1}{e^{\left(M_0-1\right)\ell-1}}$,
        $$\mathbb{P}\left(\xi\leq\varepsilon\right)\leq KC^{\ell}\varepsilon\left(\log{\frac{1}{\varepsilon}}\right)^{\ell}.$$ 
   
    Take $\varepsilon\leq\frac{1}{e^{\left(M_0-1\right)\ell-1}}$ and $q=1-\frac{\ell}{\ell+\log{\frac{1}{\varepsilon}}}$. Note that
        $$\frac{\ell}{\ell+\log{\frac{1}{\varepsilon}}}\leq\frac{\ell}{\ell+\left(M_0-1\right)\ell}=\frac{1}{M_0},$$
    then $1-\frac{1}{M_0}\leq q<1$. By Markov's inequality,
    \begin{align*}
        \mathbb{P}\left(\xi\leq\varepsilon\right)&\leq\varepsilon^q\frac{\left(K\ell^{\ell}\right)^q}{(1-q)^{q\left(\ell-1\right)+1}}\\
        &=\left(K\ell^{\ell}\right)^q\varepsilon^q\left(\frac{1}{1-q}\right)^{q(\ell-1)+1}\\
        &\leq K\ell^{\ell}\cdot\varepsilon\cdot\exp\left\{{\frac{\ell\log{\frac{1}{\varepsilon}}}{\ell+\log{\frac{1}{\varepsilon}}}}\right\}\left(\frac{\ell+\log{\frac{1}{\varepsilon}}}{\ell}\right)^{\ell}\\
        &\leq Ke^{\ell}\varepsilon\left(\ell+\log{\frac{1}{\varepsilon}}\right)^{\ell}.
    \end{align*}
    Note that $\log{\frac{1}{\varepsilon}}\geq\left(M_0-1\right)\ell$, then
        $$\mathbb{P}\left(\xi\leq\varepsilon\right)\leq K \left(\frac{eM_0}{M_0-1}\right)^{\ell}\varepsilon\left(\log{\frac{1}{\varepsilon}}\right)^{\ell}\leq K\left(3e\right)^{\ell}\varepsilon\left(\log{\frac{1}{\varepsilon}}\right)^{\ell}.$$
\end{remark}

\end{document}